\documentclass[leqno]{amsart}
\usepackage{latexsym,amssymb,amsmath,graphics}
\def\E{{\mathbb{E}}}

\def\N{{\mathbb{N}}}
\def\P{{\mathbb{P}}}

\def\R{{\mathbb{R}}}

\def\Z{{\mathbb{Z}}}

\newcommand{\8}{\infty}
\renewcommand{\d}{\delta}
\renewcommand{\a}{\alpha}
\renewcommand{\b}{\beta}

\renewcommand{\l}{\lambda}

\newcommand{\e}{{\mathrm{e}}}
\newcommand{\eps}{\varepsilon}

\newcommand{\g}{\gamma}
\newcommand{\s}{\sigma}

\newcommand{\ov}{\overline}

\newcommand{\wh}{\widehat}
\renewcommand{\th}{\theta}

\newcommand{\RR}{\mathbb{R}}

\newcommand{\Ind}{\mathbf{1}_}

\newcommand{\Rp}{{\RR_+^*}}
\newcommand{\esp}[1]{\EE\left[#1\right]}

\newcommand{\nor}[1]{\left|#1 \right|}
\newcommand{\EE}{\mathbb{E}}

\newtheorem{thm}[equation]{Theorem}
\newtheorem{cor}[equation]{Corollary}
\newtheorem{lem}[equation]{Lemma}
\newtheorem{lemma}[equation]{Lemma}

\newtheorem{prop}[equation]{Proposition}
\theoremstyle{definition}

\numberwithin{equation}{section}

\begin{document}
\title[Invariant measure in the critical case]{ On the invariant
measure of the random difference equation $X_n=A_n X_{n-1}+ B_n$ in
the critical case}
\author[S. Brofferio, D. Buraczewski and E. Damek]
{Sara Brofferio, Dariusz Buraczewski and Ewa Damek}
\address{S. Brofferio\\ Universit\'e Paris-Sud\\
Laboratoire de Math\'ematiques\\
91405 Orsay Cedex\\ France }
\email{sara.brofferio@math.u-psud.fr}
\address{D. Buraczewski and E. Damek\\
Institute of Mathematics\\
University of Wroclaw\\
pl. Grunwaldzki 2/4\\ 50-384 Wroclaw\\ Poland\\}
\email{dbura@math.uni.wroc.pl\\ edamek@math.uni.wroc.pl}

\begin{abstract}
We consider the autoregressive model on $\R^d$ defined by the
following stochastic recursion $X_n = A_n X_{n-1}+B_n$, where
$\{(B_n,A_n)\}$ are i.i.d. random variables valued in $\R^d\times
\R^+$. The critical case,  when $\E\big[\log A_1\big]=0$, was
studied by Babillot, Bougeorol and Elie, who proved that there
exists a unique invariant
 Radon measure $\nu$ for the Markov chain
$\{ X_n \}$. In the present paper we prove that the weak limit of
properly dilated measure $\nu$ exists and defines a homogeneous
measure on $\R^d\setminus \{0\}$.

\end{abstract}

\thanks{
This research project has been partially supported
 Marie Curie Transfer of Knowledge Fellowship {\it Harmonic Analysis,
Nonlinear Analysis and Probability} (contract number MTKD-CT-2004-013389).
S. Brofferio was also supported by CNRS project ANR - GGPG (R\'ef. JCJC06 149094).
D.~Buraczewski and E.~Damek were also supported by KBN
grant   N201 012 31/1020.
}
\maketitle

%\marginpar{\hspace{0pt}the table of contetents is just temorary}
%\tableofcontents

\section{Introduction and the main result}

We consider the autoregressive process on $\R^d$:
\begin{equation}
\label{process}
\begin{split}
X_0^x&=x,\\
X_n^x&=A_n X_{n-1}^x + B_n,
\end{split}
\end{equation}
where the random pairs $\{(B_n,A_n)\}_{n\in \N}$ valued in $\R^d\times \R^+$ are
independent, identically distributed (i.i.d.) according to a given
probability measure $\mu$. Markov chain \eqref{process} occurs in
various applications e.g. in biology and economics, see
\cite{BBE,RS} and the comprehensive bibliography there.

It is convenient to define $X_n$ in the group language. Let $G$ be
the ``$ax+b$'' group, i.e. $G=\R^d\rtimes\Rp$, with multiplication
defined by $(b,a)\cdot(b',a') = (b+ab',aa')$. The group $G$ acts
on $\R^d$ by $ (b,a)\cdot x  = ax+b$,

 for $(b,a)\in G$ and $x\in\R^d$.
 For each $n$, we sample the random variables $(B_n,A_n)\in G$ independently with respect to the  measure $\mu$
and we write
$
W_n = (B_n,A_n)\cdot\ldots\cdot(B_1,A_1)
$ for the left random walk on $G$. Then
$
X_n^x = W_n \cdot x$.

\medskip

The Markov chain $X_n^x$ is usually studied under the assumption
$\E\big[\log A_1\big] < 0$. Then, if additionally
$\E\big[\log^+|B|\big]<\8$, there is a unique stationary
probability measure $\nu$ \cite{K}, i.e. the measure $\nu$ on
$\R^d$ satisfying
\begin{equation*}
%\label{c44}
\mu *_G\nu(f) = \nu(f),
\end{equation*}
for any positive measurable function $f$. Here
$$ \mu *_G\nu(f) = \int_G\int_{\R^d} f(ax+b)\nu(dx)\mu(db\,da). $$
In a number of papers \cite{K,Grin1,L,Go,Gu, BD}, under some
additional assumptions, behavior of the tail of $\nu$ was studied.
Roughly speaking, it is known that, if there is a positive constant $\a$ such that $\E[A_1^\a]=1$,
%and $\mu$ has moments slightly bigger then $\a$,
 then
$$
\nu\big(\{ t:\;|t|>z \}\big)\sim C z^{-\a}\quad\mbox{ as } z\to +\8,
$$
 In \cite{BD} not only the size of
the tail is studied but also the asymptotic behavior of $\nu $ at
infinity. It is proved there that the weak limit of $ z^{\alpha }
\d_{(0,z^{-1})}*_G\nu(K)$, when $z\to \infty$ exists and it is a
Radon measure homogeneous of degree $\alpha $.

\medskip

Here we study the critical case, when $\E\big[\log A_1\big] =0$.
Then $X_n$ has no invariant
probability measure. However it was proved by
Babillot, Bougerol and Elie \cite{BBE} that under the following hypotheses
\begin{itemize}
  \item $\P[A_1=1]<1$ and $\P[A_1x+B_1=x]<1$ for all $x\in\R^d$,
  \item $\esp{(|\log A_1|+\log^+|B_1|)^{2+\eps}}<\8$,
  \item $\esp{\log A_1}=0$.
\end{itemize}
there exists a unique (up to a constant factor) invariant Radon
measure $\nu$  (see also \cite{Be,Br}). We are going to say that
$\mu$ satisfies hypothesis {\bf (H)} if all the assumptions above
are satisfied.

\medskip

Our aim is to study behavior of $\nu$ at infinity, i.e. to
understand how the measure $\nu(zK)=\d_{(0,z^{-1})}*_G\nu(K)$ of
the compact set $K$ dilated by $z>0$ behaves as $z$ goes to
infinity.
The known results  concern mainly the one dimensional setting and
they have been proved under quite restrictive hypotheses. The
first estimate of $\nu$ is given in \cite{BBE}, where it is
assumed that $d=1$ and that the closed semigroup generated by the
support of $\mu$ is the whole group $G$. Then for every $\a<\b$
$$
\nu\big((\a z,\b z]\big) \sim \log(\b/\a)\cdot L(z) \quad \mbox{ as } z\to +\8,
$$
where $L$ is a slowly varying function.

More recently the second author has proved, in \cite{B}, that $L$
is constant, provided the measure $\mu_A$ (the projection of $\mu$
onto $\Rp$) is spread-out and $\mu$ has finite small moments.
(as defined in \eqref{h1})
Moreover he has proved the positivity of the constant $C$ in the
special case $B_1>\eps$ a.s. for $\eps>0$.

%Up to our knowledge
The only results concerning the multidimensional critical case
have been obtained in \cite{DH, BDH} in a very particular setting,
when the measure $\mu$ is related to a differential operator. More
precisely, let $\{\mu_t\}$ be the one parameter semigroup of
probability measures, whose infinitesimal generator is a second
order elliptic  differential operator on $\R^d\times \Rp$. Then
there exists a unique Radon measure $\nu $ that is
$\mu_t$-invariant, for any $t$. Moreover, $\nu$ has a smooth
density $m$ such that
$$
m(zu) \sim C(u) z^{-d} \mbox{ as } z\to +\8,
$$
for some continuous nonzero function $C$ on $\R^d\setminus\{0\}$.

\medskip

Here we consider the multidimensional situation and general
measures $\mu$. We prove that the family of measures
$\d_{(0,z^{-1})}*_G \nu$ has a weak limit as $z$ tends to $+\8$,
the limit measure is homogeneous and so it has a radial
decomposition, analogous to the one obtained in the contracting
case \cite{Gu,BD}. For that we do not need hypotheses concerning
the support of $\mu$ and nonsingularity of the measure $\mu_A$.
Moreover we prove nondegeneracy of the limit.

\medskip

To state our main result we need some further definitions.
Let $G(\mu_A)$ be the closed subgroup of $\Rp$ generated by the support of $\mu_A$.
Since $\P[A_1=1]<1$ we have two possibilities, either $G(\mu_A)=\Rp$, in this case $\mu_A$ is said aperiodic
or $G(\mu_A)$ is countable and $\mu_A$  is said periodic, then $G(\mu_A)=\langle e^p\rangle=\{ e^{np}:\; n\in \Z \}$
for some $p>0$.

%i.e. $G(\mu_A)=\langle \e^p \rangle = \{  \e^{np}:\; n\in \Z\}$,
%for $p>0$.
%Let $l(da)$ be the Haar measure of $G(\mu_A)$, i.e. $l(da)=\frac{da}a$
%if $G(\mu_A)=\Rp$ or $l(da)$ is the counting measure, otherwise.
%Next define $\Sigma^{d-1}=S^{d-1}=\{|x|=1\}$ if $G(\mu_A)$ is aperiodic
% and $\Sigma^{d-1}=\{x:\; 1\le |x|< \e^p \}$ if $G(\mu_A)$ is periodic.
Given a unit vector $w$ in $\R^d$ let $\R w$
 be the line
generated by $w$ and let $\pi_w(x)$ be the orthogonal projection
of $x$ onto $\R w$;

 We will say that hypothesis {\bf (G)} is satisfied  if
 there exists an affine subspace $W\subset \R^d$ of dimension $d-1$
 and a unit vector $w$ perpendicular to $W$ such that
\begin{itemize}
\item
 the half-space $W+\R^+ w$ is $\mu$-invariant;
%\item the projection of the action of the support of $\mu$ onto $\R w=\R w$, the
%subspace of $\R^d$ generated by $w$, has no fixed points, i.e.
%$\P[\pi_w(A_1x+B_1)=\pi_w(x)]<1$ for every $x\in \R^d$, where
%$\pi_w(x)$ is the orthogonal projection of $x$ onto $\R w$;

\item the projection of the action of the support of $\mu$ onto $\R w$,
 has no fixed points, i.e. $\P[\pi_w(A_1x+B_1)=\pi_w(x)]<1$
 for every $x\in \R^d$;

\item
the following integral condition holds:
$$
\E\Big[ \log^- \big| \pi_w (B_1 + (1-A_1)w_0)  \big|
\Big]<\8,
$$
where $w_0$ is the multiple of $w$ such that $-w_0\in W$.
\end{itemize}
Notice that hypothesis ({\bf G}) is fulfilled if e.g. one of  coordinates of $B_1$, say $\pi_1(B_1)$
is positive a.e., $\E\big[\log^-|\pi_1(B_1)|\big]<\8$ and the action of $\mu$ on the first coordinate
has no fixed points.
\medskip

The main results of the paper are the following

\begin{thm}
\label{mthm1} Assume that hypothesis {\bf (H)} is satisfied, the measure $\mu_A$
is aperiodic and
either
\begin{equation}
\label{h1} \E \big[ A^\d + A^{-\d}+|B|^\d \big] < \8 \mbox{ for some  $\d>0$, that is $\mu$ \emph{has small moments}},\quad\quad\quad\quad\quad\quad\quad
\end{equation}
or
\begin{equation}
\label{h1'} \E \big[ (|\log A| + \log^+|B|)^{4+\eps}\big] < \8 \mbox{ for some $\eps>0$ and
hypothesis {\bf (G)} holds.}\quad\quad\quad\quad\quad\quad
\end{equation}
Then there exists a probability  measure  $\Sigma$ on  $S^{d-1}=\{|u|=1\}$\! and a positive number $C_+$ such that
the measures $\d_{(0,z^{-1})}*_G\nu$ converge weakly on $\R^d\setminus\{0\}$ to $C_+ \Sigma \otimes \frac{da}a$ as $z\to+\8$,
that is
$$
\lim_{z\to +\8}\int_{\R^d}\phi(u z^{-1})\nu(du) = C_+  \int_{\Rp}\int_{S^{d-1}}\phi(aw)\Sigma(dw)\frac{da}a
$$
%\end{equation}
for every function $\phi\in C_c(\R^d\setminus\{0\})$.

In particular for every $\a<\b$
\begin{equation}
\label{rlimit}
\lim_{z\to \8}\nu\big(u:\; \a z<|u| < \b z \big) = C_+ \log\frac\b\a .
\end{equation}
\end{thm}

\begin{thm}
\label{mthm2}
If hypothesis {\bf (H)} is satisfied, the measure $\mu_A$ is periodic, i.e. $G(\mu_A)=\langle e^p \rangle$ and either
\eqref{h1} or \eqref{h1'} holds, then there exists a positive constant $C_+$ such that
$$ \lim_{z\to\8} \nu\big(u:\;  z<|u| < e^{np} z \big) = nC_+,$$
for every $n\ge 1$.
\end{thm}

%We would like to observe that
%%, as far as we know, the moment hypothesis we require, are in some sense weaker
%%than the one required in the contracting case. That could suggest that
%our techniques could lead to estimates of
% the tail behavior in the contracting case  when the  harmonic exponential $\a$ does not exist.
%Some partial results are obtained under weaker conditions than in the main theorem, e.g.
%for periodic measures.
The first estimate of the behavior of the tail of $\nu$    is given in Section
\ref{sect-upper-bound} under the very mild hypothesis {\bf (H)}.
In Theorem \ref{thm-upp-bound}, we prove
that    $\d_{(0,z^{-1})}*_G\nu(K)$ is smaller than $C_K L(z)$, for
all compact sets $K$ and a slowly varying function $L$, i.e. the
family of measures $\d_{(0,z^{-1})}*_G\nu/L(z)$ is weakly compact.
We also show some invariance properties of the accumulation points
and an upper-bound for the measure $\nu$ that implies $\nu(du)$
integrability of the function $(1+|u|)^{-\g}$ for every $\g>0$.
Moreover, if additionally hypothesis {\bf (G)} is assumed, then
(Proposition \ref{prop2}) says that the slowly varying function
$L$ is dominated by the logarithm function,
 and for every $\g>0$ the function $\log^{-(1+\g)}(2+|u|)$
is $\nu(du)$ integrable. All the estimates are crucial in the
remaining  part of the proof.

\medskip

Next we reduce the problem of describing the tail of $\nu$ to
study asymptotic behavior of positive solutions of the Poisson
equation, as in \cite{Bu}. More precisely, given  a positive
$\phi\in C_c(\R^d\setminus\{0\})$ we define the function
\begin{equation}
%\label{def}
f_{\phi}(x) = \d_{(0,e^{-x})}*_G\nu(\phi)=\int_{\R^d}\phi(\e^{-x}
u)\nu(du).
\end{equation}
on $\R$. Let $\ov \mu$ be the law of $-\log A_1$ i.e. for a Borel
set $U$
\begin{equation} \label{overmu} \ov \mu (U) = \mu_A(\{x:\;
-\log x\in U\}).
\end{equation}
  Observe that the mean of $\ov\mu$ is
equal to 0. The convolution on $\R$ of $\ov\mu$ and a function $f$
is  $\ov\mu*_\R f(x)= \int f(y+x)\ov\mu(dy)$ and the function
$f_{\phi}$ satisfies the Poisson equation
\begin{equation}
\label{poisson}
\ov\mu*_\R f_{\phi}(x) = f_{\phi}(x)+\psi_{\phi}(x),\qquad x\in\R,
\end{equation}
for a specific function $\psi_\phi$. $\psi_\phi$ posses some
regularity properties and it is easier to study than $f_\phi$. The
main problem can be formulated as follows: given a function
$\psi_{\phi}$ describe the behavior at infinity of positive
solutions of the Poisson equation. An answer to this rather
classical question has been given by Port and Stone \cite{PS2},
under the hypothesis that $\ov\mu$ is spread out.  To avoid
this restriction, in Section \ref{sect-poisson} and Appendix
\ref{port-stone}, we generalize Port and Stone
 results to a generic centered measure $\ov\mu$ on $\R$. Namely, we introduce a class of functions ${\mathcal F}(\ov\mu)$,
 larger than the one considered in \cite{PS2},
such that if $\psi\in {\mathcal F}(\ov\mu)$ then the recurrent
potential kernel  $A\psi$ is well defined and using it we describe
asymptotic of bounded solutions of the corresponding Poisson
equation. We also give, using results of Baldi \cite{Ba}, an
explicit description of the solutions of the Poisson equation in
term of recurrent potential.

\medskip
The next step is to prove that the function $\psi_{\phi}$ belongs to ${\mathcal F}(\ov\mu)$, that in fact can be done only for
 very particular functions $\phi$. However this class is sufficient to deduce our main result (Section \ref{sect-existence}).
% We also prove a partial result in the case when the measure $\mu_A$ is not necessarily aperiodic (see Proposition \ref{mprop}).
Finally in Section \ref{sect-positivity} we prove  that the limit
of $\nu\big( \a z<|u| < \b z \big)$ is strictly positive and the
only hypothesis needed for this results is condition \textbf{(H)}.

\medskip
%3 razy of to za duzo
%Let us mention briefly that
%Our results are closely related to the problem
%of describing tails of fixed points of smoothing transformations in
%'the boundary case', that has still not been completely understood.
Our results can be aplied to study fixed points of smoothing transformations.
More precisely we are able to describe their asymptotic behavior at infinity in the so called boundary case.
Roughly speaking,
let $\{A_n\}_{n\in\N}$ be a sequence of random positive numbers such that  only first $N$
of them are nonzero, where $N$ is some random number, finite a.s.
Given a  random variable $Z$ and a sequence $\{Z_n\}_{n\in\N}$ of i.i.d. copies of $Z$ independent
both on $N$ and  $\{A_n\}_{n\in\N}$, we define new random variable $Z^*=\sum_{i=1}^NA_iZ_i$
and the map $Z\to Z^*$ is called smoothing transformation. The random variable $Z$ is
said to be fixed point  if $Z^*$ has the same distribution as $Z$.
There exists an extensive literature, where the problems of existence,
uniqueness and asymptotic behavior of Z  were
studied, see e.g. Durrett and Liggett \cite{DL}, Liu \cite{Liu00},
Biggins and Kyprianou \cite{BK}. The case of solutions having finite mean
has been completely described, however in the boundary case, when $Z$
has infinite mean all the results concerning asymptotic behavior of $Z$
are expressed in terms of the Laplace transform. Applying techniques
introduced by Guivarc'h \cite{Gui} and Liu \cite{Liu00} and Theorem \ref{mthm1}
one can prove (of course under appropriate assumptions) that
$\lim_{x\to\8}x\P[Z>x]$ exists and is positive. A complete proof
of this fact will be the subject of a future work.

\medskip

The authors are grateful to Jacek Dziuba\'nski and Yves Guivarc'h for
stimulating discussions on the subject of the paper.

\section{Upper bound}
\label{sect-upper-bound} The goal of this section is to prove  a
preliminary estimate of the measure
 $\nu$ at infinity. We first prove that, under very mild hypothesis {\bf (H)} on the
 measure $\mu$,
 %(that is when even we can guarantee the existence of the $\mu$-stationary(?) measure)
the tail measure of compact set $\d_{(0,z^{-1})}*\nu(K)$ is
bounded by a slowly varying function $L$, that is a
 function
 on $\Rp$ such that $\lim_{z\to+\infty} L(a z)/L(z)=1$ for all $a>0$.
 An important property of such  functions is that they grow very slowly, namely slower than $z^\gamma$ for any $\gamma>0$.

  Under the hypothesis {\bf (G)} we will  prove in the next subsection a stronger result: any bounded function that is integrable, at infinity,
  with respect to the measure $\frac{da}{a}$ (i.e. the Lebesgue measure of $\Rp$)  is $\nu$-integrable. In particular
  this implies that $L(z)$ is bounded by $\log z$.
\subsection{ Generic case}
\begin{thm}\label{thm-upp-bound}
If hypothesis {\bf (H)} is fulfilled,
then there exists a positive slowly varying function $L$ on $\Rp$ such that  the normalized family of measures on $\R^d\setminus\{0\}$
  $$\frac{\d_{(0,z^{-1})}*_G\nu}{L(z)}$$ is weakly compact  for $z\geq 1$. That is
  $\frac{\delta_{(0,z^{-1})}*_G \nu(\phi)}{L(z)}$ is bounded  for all  $\phi$ bounded and compactly supported.
  Thus $(1+\nor{x})^{-\gamma}\in L^1(\nu)$ for all $\gamma>0$.
 Furthermore, all accumulation points $\eta$ are non null and invariant with respect to the group $G(\mu_A)$,  that is
  $$\delta_{(0,a)}*_G\eta=\eta\quad \forall a\in G(\mu_A).$$
\end{thm}
Theorem \ref{thm-upp-bound} is a partial generalization of
Proposition 5.2 in \cite{BBE}. It is proved there that this family
of measures converges to the Lebesgue measure of $\Rp$, if $d=1$
 and the closed semi-group generated by the support of  $\mu$ is the whole group $G$.

\medskip

First we prove  that, since  the support of the measure $\mu$
contains contracting and dilating elements, there exists a
compactly supported function $r$ such that the quotient family
$\frac{\delta_{(0,z^{-1})}*_G \nu}{\delta_{(0,z^{-1})}*_G \nu(r)}$
is weakly compact, that is
\begin{prop}\label{prop-weak-bound} Under hypothesis \textbf{(H)},
there exists a bounded compactly supported function $r$ such that
$\d_{(0,z^{-1})}*_G\nu(r)$ is strictly positive  for all $z\geq1$.
Furthermore, for every compact set $K$
 there is a  positive constant $C_K$  such that
\begin{equation}\label{eq-weak-bound}
\d_{(0,z^{-1})}*_G\nu(K)\leq C_K \d_{(0,z^{-1})}*_G\nu(r)\quad \forall\ z\geq 1.
\end{equation}

\end{prop}
\begin{proof}
For all real numbers $\a$ and $\b$, consider the annulus
$$C(\a,\b)=\{u\in\R^d \ |\ \a\leq |u|\leq \b\}.$$
Observe that if either $\a>\b$ or $\b<0$ the set $C(\a,\b)$ is
void. It is easy to check that for all $(b,a)\in G$ the following
implication holds
$$u\in C\left(\frac{\a +|b|}{a},\frac{\b -|b|}{a}\right) \quad \Rightarrow \quad au+b\in C(\a,\b)$$
Let $U$ be an open set in $G$ and $n\in \N$. Since $\nu$ is
invariant with respect to $\mu^{*n}$, we have
\begin{align}\label{eq-inf-gC}
  \d_{(0,z^{-1})}*_G\nu(C(\a,\b))&=\int_{\R^d}\int_G \Ind{C(\a,\b)}(z^{-1}(au+b))\ \mu^{*n}(db\,da)\nu(du) \\
  &\nonumber\geq\int_{\R^d}\int_G \Ind{C(z\a,z\b)}(au+b)\Ind{U}(b,a)\ \mu^{*n}(db\,da)\nu(du) \\
  &\nonumber\geq \mu^{*n}(U)\nu\left(C\left(\max_{(b,a)\in U}\frac{\a z+|b|}{a},\min_{(b,a)\in U} \frac{\b z -|b|}{a}\right)\right).
\end{align}

First we prove that there exists a sufficiently large $R>0$ such
that $\d_{(0,z^{-1})}*_G\nu(C(1/R,R))$ is strictly positive for
all $z\geq 1$. By hypothesis {\bf (H)} the support of $\mu$
contains at least two elements
$g_+=(b_+,a_+)$ and $g_-=(b_-,a_-)$ with $a_+>1>a_-$.\\
Fix $z\geq 1$ and take $n\in \N$ such that $a_+^{n-1}\leq z\leq
a_+^n$. Clearly, if $g^n=(b(g^n),a(g^n))$ is the $n$-th power of
an element $g=(b,a)\in G$ then
$$ a(g^n)=a^n \quad \mbox{and} \quad b(g^n)=\sum_{i=0}^{n-1}a^ib=\frac{a^n-1}{a-1}b.$$
Consider the $\d$-neighborhood  of $g^n$
$$U_\d(g^n)=\left\{(b,a)\in G| \e^{-\d} < a a(g)^{-n}<\e^\d \mbox{ and } |b-b(g^n)|<\d\right\}.$$
Observe that   $\mu^{*n}(U_\d(g^n_+))>0$ for all $\d>0$ and for $(b,a)\in U_\d(g^n_+)$
\begin{align*}
&\frac{z/R +|b|}{a}\leq \frac{a_+^n/R + |b(g^n_+)|+\d}{\e^{-\d}a_+^n}
\leq \e^\d \Big(\frac{1} R+\frac{|b_+|}{a_+-1}+\d\Big)=:\a_R\\
&\frac{ R  z-|b|}{a}\geq \frac{ R  a_+^{n-1}- |b(g^n_+)|-\d}{\e^{\d}a_+^n}
\geq \e^{-\d} \Big( \frac{R}{a_+} -\frac{|b_+|}{a_+-1}-\d\Big)=:\b_R.
\end{align*}
Since $\nu$ is a Radon measure with infinite mass, its support cannot be compact.
Thus, for a fixed $\d$, there exits a   sufficiently large $R$ such that:
$\nu(C(\a_R,\b_R))>0.$ Then by \eqref{eq-inf-gC}:
\begin{equation}
\label{eq-CR}
\d_{(0,z^{-1})}*_G\nu(C(1/R,R))\geq \mu^{*n}(U_\d(g_+^n))\nu(C(\a_R,\b_R))>0
\end{equation}
for all $z\geq1$.

\medskip
For $R>2$ consider the compact sets
$K^n_\pm=C( 2a_\pm^{-n}/R, a_\pm^{-n}R/2).$ Observe  that for
$\d<\log(4/3)$, $(b,a)\in U_\d(g^n_\pm)$ and $z>z^n_\pm:=2R(|b(g^n_\pm)|+ \d)$ :
\begin{align*}
&\frac{z /R +|b|}{a}
\leq z \e^{\d}\frac{1/R +z^{-1}(|b(g^n_\pm)|+\d)}{a^n_\pm}\leq z \frac{2a_\pm^{-n}}{R} \left(\e^{\d} \frac {1 +z^{-1}R(|b(g^n_\pm)|+\d)} 2\right)  \leq z \frac{2a_\pm^{-n}}{R}\\
&\frac{zR-|b|}{a}
\geq z \e^{-\d}\frac{R-z^{-1}(|b(g^n_\pm)|+\d)}{a^n_\pm}
\geq z \frac{a_\pm^{-n}R}{2} \cdot 2\e^{-\d}\left(  1 -z^{-1}\frac{| b(g^n_\pm)|+\d}{R} \right)
\geq z \frac{a_\pm^{-n}R}{2}.
\end{align*}
Thus
 by
\eqref{eq-inf-gC}:
$$\d_{(0,z^{-1})}*_G\nu(C(1/R,R))\geq \mu^{*n}(U_\d(g_\pm^n))\nu(C(z\ 2a_\pm^{-n}/R, z\ a_\pm^{-n}R/2))=C_{K^n_\pm}^{-1}\d_{(0,z^{-1})}*_G\nu(K^n_\pm)$$
for all $z>z^n_\pm$. Since $\d_{(0,z^{-1})}*_G\nu(C(1/R,R))>0$,
the above inequality holds in fact for all $z\geq1$, possibly with
a bigger constant $C_K$ and sufficiently large $R$.

We may assume that $R>2\max\{a_+,1/a_-\}$, then the family of sets $K^n_\pm$ covers $\R^d\setminus \{0\}$. Take
 any function $r\in C_c(\R^d\setminus\{0\})$ such that $r(u)\geq {\bf 1}_{C(1/R,R)}(u)$.
Then a generic compact set $K$ in $\R^d\setminus\{0\}$ is covered by a finite number of compacts  $\{K_i\}_{i\in I}$ of the type $K^n_\pm$  and
\begin{align*}
  \d_{(0,z^{-1})}*_G \nu(K)\leq \sum_{i\in I}\d_{(0,z^{-1})}*\nu(K_i)\leq \left(|I|\max_{i\in I}C_{K_i}\right)\ \d_{(0,z^{-1})}*_G \nu(r)
\end{align*}
for all $z\geq 1$. Moreover, in view of \eqref{eq-CR}, $\d_{(0,z^{-1})}*_G\nu(r)$ is strictly positive
for all $z\ge 1$, which finishes the proof.
\end{proof}

Now  we prove that  $\mu$-invariance of $\nu$ implies that the
accumulation points are invariant by the action of $G(\mu_A)$,
namely we have
\begin{lem}\label{lem-inv-lim}
  Suppose that there exists a function $L(z)$  such that $$\frac{\delta_{(0,z^{-1})}*_G \nu}{L(z)}$$ is weakly
  compact when $z$ goes to $+\infty$, then the accumulation points $\eta$ are invariant for the action of  $G(\mu_A)$.
\end{lem}
\begin{proof}
Let $\eta$ be a limit measure i.e. there is a subsequence
$\{z_n\}$ such that
$$\lim_{n\to\infty}\frac{\d_{(0, z_n^{-1})}*_G\nu(\phi)}{L( z_n)}=\eta(\phi) \quad \forall \phi\in C_c(\R^d\setminus\{0\}).$$
Fix a function $\phi\in C^1_c(\R^d\setminus\{0\})$ and observe
that for all $(b,a)\in G$ there is a compact set $K=K(b)$ and a
constant $C$ such that $$|\phi( z^{-1}(au+b))-\phi( z^{-1}(au))|<
C| z^{-1}b|{\bf 1}_K( z^{-1}(au))$$ for all $z>1$ and $u\in \R^d$.

\medskip

We claim that the function
$$h(y)=\d_{(0,y)}*_G\eta(\phi)=\lim_{n\to\infty}\frac{\d_{(0, z_n^{-1}y)}*_G\nu(\phi)}{L( z_n)}$$ on $\R^+$ is $\mu_A$-superharmonic.
Indeed, take a function $\psi\in C_c(\R^d\setminus\{0\})$ such
that $\psi \ge {\bf 1}_{a^{-1}K}$, then
\begin{eqnarray*}
\lim_{n\to\8} \frac{|\d_{(0, z_n^{-1})}*_G\d_{(b,a)}*_G\nu(\phi)-\d_{(0, z_n^{-1})}*_G\d_{(0,a)}*_G\nu(\phi)|}{L( z_n)}
 &\leq& \lim_{n\to\8} \frac{C| z_n^{-1}b| \nu( a^{-1}z_n K)}{L( z_n)}\\
 &\leq& \lim_{n\to\8} \frac{C| z_n^{-1}b| \d_{(0,z_n^{-1})}*_G \nu( \psi)}{L( z_n)}\\
&=&  C \eta(\psi)\cdot \lim_{n\to \8} | z_n^{-1}b|= 0,
\end{eqnarray*}
hence
\begin{eqnarray*}
\int_G h(ay)\mu_A(da) &=& \int_G \lim_{n\to\8}\frac{\d_{(0, z_n^{-1}y)}*_G\d_{(0,a)}*_G\nu(\phi)}{L( z_n)}\mu(db\,da) \\
   &=& \int_G \lim_{n\to\8}\frac{\d_{(0, z_n^{-1}y)}*_G\d_{(b,a)}*_G\nu(\phi)}{L( z_n)}\mu(db\,da)\\
   &\leq& \lim_{n\to\8}\frac{ \d_{(0, z_n^{-1}y)}*_G\mu*_G\nu(\phi)}{L( z_n)}\quad\mbox{ by Fatou's Lemma}\\
   &=& \lim_{n\to\8}\frac{\d_{(0, z_n^{-1}y)}*_G\nu(\phi)}{L(z_n)}=h(y)
\end{eqnarray*}
Since $h$ is positive,
then by Choquet-Deny theorem $h(ay)=h(y)$ for every $a\in G(\mu_A)$, that is  $$\delta_{(0,a)}*_G\eta(\phi)=\eta(\phi), \qquad \forall \phi\in C_c(\R^d\setminus\{0\}).$$
\end{proof}

\begin{proof}[Proof of Theorem \ref{thm-upp-bound}]
Let $r$ be the function introduced in Proposition \ref{prop-weak-bound} and take any
$R\in C_c(\R^d\setminus \{0\})$ such that $R>r$. Let $L(z)=\d_{(0,z^{-1})}*_G\nu(R)$, so that,
by Proposition \ref{prop-weak-bound},
$\d_{(0,z^{-1})}*_G\nu/L(z)$ is weakly compact when $z$ goes to $+\8$.
It  remains to prove that $L$ is a slowly varying function.
Observe that
$$L(az)=\int_{\R^d} R(z^{-1}a^{-1}u)\nu(du)=\delta_{(0,z^{-1})}*_G\nu(R_a)$$
where $R_a(u)=R(a^{-1}u)$.

Since $R_a$ has compact support thus $L(az)/L(z)$ is bounded. Let
$z_n$ be a sequence such that both $L(az_n)/L(z_n )$ and
$\d_{(0,z_n^{-1})}*_G\nu/L(z_n)$ converge i.e. there is a number
$l$ and a measure $\eta $ such that
\begin{align*}
l=\lim_{n\to\infty}\frac{L(az_n)}{L(z_n)}=\eta(R_a)=\delta_{(0,a^{-1})}*_G\eta(R).
\end{align*}
Observe that, since $\eta(R)=1$, then for every $a$  such that $\delta_{(0,a^{-1})}*_G\eta(R)=\eta(R)$
$$\lim_{z\to+\infty}\frac{L(az)}{L(z)}=1.$$
By Lemma  \ref{lem-inv-lim}, $\eta$ is invariant under the action
of closed group $G(\mu_A)$ generated by the support of $\mu_A$ and
the proof in the case $\mu_A$ aperiodic is completed.

If $G(\mu_A)=\langle e^p\rangle$,  consider the function
$$R(u)=\int_\Rp\Ind{[\e^{-p},\e^{p})}(t)r({u}/{t})\frac{dt}{t}.$$
An easy argument shows that $R$ is in $C_c(\R^d\setminus\{0\})$
and it is bigger of some multiple of $r$. We claim that
$\delta_{(0,a^{-1})}*_G\eta(R)=\eta(R)$ for all $a\in\Rp$ not only
for $a\in G(\mu_A)$. In fact, let $\e^{Kp}\in G(\mu_A)$ such that
$\e^{Kp}>a \e^{p}$ then
\begin{align*}
\delta_{(0,a^{-1})}*_G\eta(R)&=
     \int_{\R^d} \int_{\Rp} \Ind{[a\e^{-p},a\e^{p})}(t)r({u}/{t})\frac{dt}{t}\eta(du)\\
    &=\int_{\R^d} \int_{\Rp}  \Big( \Ind{[a\e^{-p},\e^{Kp})}(t)-\Ind{[a\e^{p},\e^{Kp})}(t)\Big) r(u/t)\frac{dt}{t}\eta(du)\\
    &=\int_{\R^d} \int_{\Rp}\!\! \bigg(   \Ind{[a\e^{-Kp},\e^{p})}(t) r\big({\e^{-(K-1)p}u}/{t}\big)
    \!-\! \Ind{[a\e^{-Kp},\e^{-p})}(t) r\big({\e^{-(K+1)p}u}/{t}\big)\!\!\bigg) \frac{dt}{t}\eta(du)\\
    &=\int_{\R^d} \int_{\Rp}\!\! \bigg(   \Ind{[a\e^{-Kp},\e^{p})}(t) r\big(u/{t}\big)
    \!-\! \Ind{[a\e^{-Kp},\e^{-p})}(t) r\big(u/{t}\big)\!\!\bigg) \eta(du) \frac{dt}{t} \quad \mbox{}\\
    &=\int_{\R^d} \int_{\Rp} \Ind{[\e^{-p},\e^{p})}(t)r(u/t)\frac{dt}{t}\eta(du)=\eta(R),
    \end{align*}
since $\eta$ is $G(\mu_A)$-invariant.
\end{proof}

The following lemma will be used in the sequel to give bounds for
integrals against $\nu$. The statement holds for any Radon measure
$\rho $ on $\R^d\setminus\{0\}$ for which we can control
  the growth of $\delta_{(0,z^{-1})}*_G \rho(K)$.

\begin{lem}\label{lem-nu-leb} Let $\rho$ be a Radon measure and $l(z)$ a nondecreasing function such that
$$ \rho(z<|u|\leq z\e)\leq C_1\,l(z) $$
for a constant $ C_1$ and for every $ z\in \Rp.$
  Then for all  $M>0$ and all non negative functions $f$
  $$\int_{|u|\geq M}f(|u|)\rho(du)\leq C_1\int_{\e^{-1}M}^\infty f(a)l(a)\frac{da}{a},\ \mbox{if}\ f\ \mbox{is nonincreasing}$$
  and
  %for all non negative and nondecreasing function $f$
  $$\int_{0<|u|\leq M}f(|u|)\rho(du)\leq C_1 \int_0^{eM} f(a)l(a)\frac{da}{a}, \ \mbox{if}\ f\ \mbox{is nondecreasing}.$$
In particular under hypothesis \textbf{(H)},
$$ \int_{\R^d}\frac 1{1+|u|^\g}\nu(du)<\8
$$
for all $\g>0$.
\end{lem}
\begin{proof}
If $f$ is nonincreasing
\begin{eqnarray*}%\label{nu<leb}
\int_{|u|\geq M}f(|u|)\rho(du)&=&\sum_{n=0}^\8 \int_{M\e^n \le |u| < M\e^{n+1}} f(|u|)\rho(du)\\&\leq& \sum_{n=0}^\8 f(M\e^n) \rho({M\e^{n-1}e \le |u| < M\e^{n-1}\e^2})\\
&\leq & C_1 \sum_{n=0}^\8 f(M\e^n)l(M\e^{n-1})\\
%&=& C_1 \sum_{n=0}^\8 f(M\e^n)l(M\e^{n-1})\int_{M\e^{n-1}}^{M\e^n}\frac{da}{a}\\
&\leq& C_1 \sum_{n=0}^\8 \int_{M\e^{n-1}}^{M\e^n}f(a)l(a)\frac{da}{a}\\  &=& C_1\int_{\e^{-1}M}^\infty f(a)l( a)\frac{da}{a}.
\end{eqnarray*}
Exactly in the same way we prove the second part of the Lemma.

We can apply this result to $\nu$, since any slowly varying function $L(z)$ is smaller of a multiple of $z^\g$ for all $\g>0$.

%In the same way if $f$ is nondecreasing
%Let $C_2$ be such that
%$$ \rho(z\e^{-1}<|u|\leq z)\leq C_2\,l(z) \quad \mbox{ for every } z\in \Rp$$
%then
%\begin{eqnarray*}%\label{nu<leb}
%\int_{0<|u|\leq M}f(|u|)\rho(du)&=&\sum_{n=0}^\8 \int_{M\e^{-(n+1)} \le |u| < M\e^{-n}} f(|u|)\rho(du)
%\\&\leq& \sum_{n=0}^\8 f(M\e^{-n}) \rho({M\e^{-n}\e^{-1} \le |u| < M\e^{-n}})\\
%&\leq & C_2 \sum_{n=0}^\8 f(M\e^{-n})l(M\e^{-n})\\
%&=& C_2 \sum_{n=0}^\8 f(M\e^{-n})l(M\e^{-n})\int_{M\e^{-n}}^{M\e^{-n+1}}\frac{da}{a}\\
%&\leq& C_2 \sum_{n=0}^\8 \int_{M\e^{-n}}^{M\e^{-n+1}}f(a)l(a)\frac{da}{a}\\
%&=& C_2\int_0^{e M} f(a)l( a)\frac{da}{a}
%\end{eqnarray*}
\end{proof}

\subsection{Upper bounds under hypothesis {\bf (G)}}

The main result of this subsection is the following
\begin{prop}
\label{prop2} Assume that hypotheses {\bf (H)} and {\bf (G)} are
satisfied, then there exists a constant $C$ such that for every
bounded nonincreasing nonnegative function $f$ on $\R$
 \begin{equation}
\label{ineq-nu-dx}
\int_{\R^d}f(|u|)\nu(du) < C(\|f\|_\8+  \int_{1/e}^\8 f(y)\frac{dy}y)
\end{equation}

In particular for every $\eps >0$
$$
\int_{\R^d} \frac 1{\log^{1+\eps}(2+|u|)}\nu(du)<\8
$$ and for $z>1/e$
\begin{equation}
\label{nulog}
\nu\big\{ |u|<z \big\} < C (2+\log z).
\end{equation}

%Moreover if $0\in W$, then
%\begin{equation}
%\label{ineq-nu-dx}
%\int_{\R^d}f(|u|)\nu(du) < C+ C\cdot \int_{1/e}^\8 f(x)\frac{dx}x,
%\end{equation}
%for every bounded nonincreasing function $f$ on $\R^+$.
\end{prop}

Let us recall the following \cite{BBE} explicit construction of
the measure $\nu$. Define a random walk on $\R$
\begin{equation}
\label{cz1}
\begin{split}
S_0&=0,\\
S_n&=\log(A_1\ldots A_n), \quad n\ge 1,
\end{split}
\end{equation}
and consider the downward ladder times of $S_n$:
\begin{equation}
\label{cz2}
\begin{split}
L_0&=0,\\
L_n&= \inf\big\{ k>L_{n-1}; S_k< S_{L_{n-1}} \big\}.
\end{split}
\end{equation}
Let $L=L_1$.  The Markov process $\{X_{L_n}^x\}$ satisfies the recursion
$$
X_{L_n}^x = M_n X_{L_{n-1}}^x+Q_n,
$$
where $(Q_n,M_n)$ is a sequence of $G$-valued i.i.d. random
variables and $(Q_n,M_n)=_d(X_L ,\e^{S_{L}})$. We denote by
$\mu_L$ the law of $(Q_n,M_n)$. It is known that $ -\8< \E S_L <0$
and $\E[\log^+|X_L|]<\8$ (see \cite{Grin2,E}), therefore there
exists a unique invariant probability measure $\nu_L$ of the
process $\{X_{L_n}\}$ and the measure $\nu$ can be written (up to
a constant) as
\begin{equation}
\label{measure}
\nu(f) =  \int_{\R^d}\E\Big[ \sum_{n=0}^{L-1}f(X_n^x) \Big]\nu_L(dx).
\end{equation}
where $X_n^x$ is the process defined  in \eqref{process}.

\begin{proof}[Proof of Proposition \ref{prop2}]
\noindent {\bf Step 1. Assume $d=1$ and $B>0$ a.s.} First we will
prove \ref{ineq-nu-dx} in the simplest one dimensional case, when
$B>0$ a.s., i.e. the positive half-line is invariant under the
action of $\mu$. Then supports of both measures $\nu$ and $\nu_L$
are contained in $\R^+$.

\medskip

%We first  prove that there exists a constant $C$ such that for every bounded %nonincreasing function $f$ on $\R$
 %\begin{equation*}
%\label{ineq-nu-dx}
%\int_{\R^d}f(|u|)\nu(du) < C(\|f\|_\8+  \int_{1/e}^\8 f(y)\frac{dy}y)
%\end{equation*}
Notice that
\begin{eqnarray*}
\nu(f)&=&
\int_{\R^+} \E\bigg[ \sum_{n=0}^{L-1}f(X_n^x)\bigg]\nu_L(dx)\\
&=& \int_{\R^+} \E\bigg[ \sum_{n=0}^{L-1} f\big({A_1A_2\cdots A_nx+A_2\cdots A_n B_1 +\cdots+B_n}\big)\bigg]\nu_L(dx)\\
&\le& \int_{\R^+} \E\bigg[ \sum_{n=0}^{L-1}
f\big(\e^{S_n}x)\big)\bigg]\nu_L(dx).
\end{eqnarray*}
Define the stopping time
%\begin{eqnarray*}
$T = \inf \big\{n:\; S_n > 0\big\}$,
%\end{eqnarray*}
where $S_n = \sum_{k=1}^n \log A_i$. Let $\{Y_i\}$ be a sequence
of i.i.d. random variables with the same distribution as the
random variable $S_T$ (recall $0<\E S_T < \8$). Using the duality
Lemma \cite{F} (see also Lemma \ref{lem-duality}) we obtain
\begin{equation}
\label{nufdual} \nu(f) \le \int_{\R^+} \E\bigg[ \sum_{n=0}^{L-1}
f\big(\e^{S_n}x)\big)\bigg]\nu_L(dx)= \int_{\R^+} \E\bigg[
\sum_{n=0}^{\8} f\big(\e^{Y_1+\cdots +Y_n}x\big)\bigg]\nu_L(dx).
\end{equation}
Let $U$ be the potential associated with the random walk
$Y_1+\ldots+ Y_n$, i.e.
$$ U(a,b) = \E\big[\# n:\; a<Y_1+\ldots + Y_n \le b \big].
$$
By the renewal theorem $U(k,k+1)$ is bounded, thus we have
\begin{eqnarray*}
\nu(f) &\le&
\int_{\R^+}\E\bigg[ \sum_{n=0}^{\8} f\big(\e^{Y_1+\cdots +Y_n}x\big)  \bigg]\nu_L(dx)\\
 &\le& \sum_{k=0}^\8 \int_{\R^+}  U(k,k+1) f\big(\e^kx\big)\nu_L(dx)\\
 &\le& C\sum_{k=0}^\8 \int_{\R^+} f\big(\e^kx\big)\nu_L(dx).
\end{eqnarray*}
Next we divide the integral into two parts. First we assume that
$x>1$:
\begin{eqnarray*}
\sum_{k=0}^\8 \int_1^\8 f\big( \e^kx\big)\nu_L(dx)
&\le& \sum_{k=0}^\8  f\big( \e^k\big)
\le  \sum_{k=-1}^\8 \int_k^{k+1}  f\big( \e^y\big)dy\\
&=&   \int_{-1}^\8  f\big( \e^y\big)dy
=   \int_{1/e}^\8 f(y)\frac{dy}y.
\end{eqnarray*}
Secondly, for $0<x<1$ we    write
\begin{eqnarray*}
\sum_{k=0}^\8 \int_0^1 f\big( \e^kx\big)\nu_L(dx)
&\le& \int_0^1\bigg( \sum_{k=0}^{|\log x|} + \sum_{k= |\log x |}^\8\bigg)  f\big( \e^kx\big)\nu_L(dx)\\
&\le& C \|f\|_\8\int_0^1 \big|\log x\big|\nu_L(dx)+  \sum_{k=0}^\8  f\big( \e^k \big)\\
&\le& C \|f\|_\8\int_0^1\big|\log x\big|\nu_L(dx)  + \int_{1/e}^\8
f(y)\frac{dy}y.
\end{eqnarray*}
Hence to prove \eqref{ineq-nu-dx} we have to justify that the
first term above is finite. For that we use the integral condition
in hypothesis {\bf (G)}, which in this setting says that
$\E\big[|\log^-B_1|\big]<\8$. Notice  that if $x,y\in \R^+$ and
$x+y<1$ then $\big|\log(x+y)\big|
 < \big|\log x\big|$.
We write
\begin{equation}
\label{log-nuL}
\begin{split}
\int_0^1 \big|\log x\big|\nu_L(dx) &= \int \int_{ax+b<1}\big|\log (ax+b)\big|\mu_L(db\,da)\nu_L(dx)\\
 &= \int_{\R^+} \E\Big[\big|\log X_L^x \big|\cdot {\bf 1}_{\{X_L^x <1\}}\Big]\nu_L(dx)\\
&\le   \int_{\R^+} \E\bigg[\Big|\log\Big(\frac{A_1 A_2 \ldots A_L B_1}{A_1}\Big)\Big|  \bigg]\nu_L(dx)\\
 &\le  \E\big[ |S_L|+|\log B_1| + |\log A_1| \big] < \8,
\end{split}
\end{equation}
that completes the proof of \eqref{ineq-nu-dx} in this case.

\medskip

\noindent
{\bf Step 2.}
To generalize the results to higher dimensions, the key
observation  is that the measure $\nu$ can be compared with the
invariant measures for projections of the process $\{X_n\}$ onto
one dimensional subspaces. Their behavior at infinity is already
controlled.

 \medskip

Let  $w\in\R^d\setminus\{0\}$ be the unit vector as in hypothesis {\bf (G)}

 and let $\pi_w$ be the orthogonal projection on the line $\R w=\{sw\}_{s\in\R}$.
Consider the random process on the line
$$ X_n^{w,x_w} = \pi_w(X_n^{x}) = A_n X_{n-1}^{w,{x_w}}+\pi_w(B_n),$$
where $x_w=\pi_w(x)$.
Let $\mu^w$ be the law of $(\pi_w(B_1),A_1)$, then the measure $\mu^w$
satisfies hypothesis {\bf (H)}. Therefore  there exists a unique Radon measure $\nu_w$ on $\R w$,
 which is the invariant measure of the process $\{X_n^{w,x_w}\}$.

We claim that $\nu^w$ is  the projection of $\nu$ onto $\R w$ that is
\begin{equation}
\label{eq-nu-proj}
 \nu^w = \pi_w(\nu).
\end{equation}
As in \eqref{measure}, we may write
$$\nu^w(g) =  \int_{\R}\E\Big[ \sum_{n=0}^{L-1}g(X_n^{w,x}) \Big]\nu_L^w(dx),$$
for any positive function $g$  on $\R$. $\nu_L^w$ is the unique
invariant measure for $\{\pi_w(X_{L_n}^x)\}= \{X_{L_n}^{w,x_w}\}$.
Notice that $\nu_L^w$ is  projection of $\nu_L$ onto $\R w$:
\begin{equation}
\label{3.4}
 \nu_L^w = \pi_w(\nu_L).
\end{equation}
Indeed, $\pi_w(\nu_L)$ is a $\mu_L^w$-invariant probability
measure:
\begin{align*}
\pi_w(\nu_L)(g)&=\int_{\R^d}g(\pi_w(u))\nu_L(du)
=\int_{\R^d}g(\pi_w(au+b))\mu_L(db\,da)\nu_L(du)\\
&=\int_G\int_{\R^d}g(a\pi_w(u)+\pi_w(b)))\mu_L(db\,da)\nu_L(du)=\mu_L^w*\pi_w(\nu_L)(g).
\end{align*}
Then
\begin{align*}
\nu_w(g) &=  \int_{\R}\E\Big[ \sum_{n=0}^{L-1}g(X_n^{w,x}) \Big]\nu_L^w(dx)=\int_{\R^d}\E\Big[ \sum_{n=0}^{L-1}g(X_n^{w,\pi_w(u)}) \Big]\nu_L(du)\\
&=
\int_{\R^d}\E\Big[ \sum_{n=0}^{L-1}g(\pi_w(X_n^u)) \Big]\nu_L(du)=\pi_w(\nu)(g)
\end{align*}

\medskip

\noindent {\bf Step 3. General case.} Let $w_0$ be a multiple
(possibly null) of $w$, such that $0\in
W+w_0$.
 The measure $\mu_0 = \d_{(w_0,1)}*_G \mu*_G\d_{(-w_0,1)}$ satisfies
hypothesis {\bf (H)}, hence there exists a unique
$\mu_0$-invariant Radon measure $\nu_0$, and one can easily prove
that $\nu_0=\d_{w_0}*_{\R^d}\nu$. $\mu_0$ satisfies hypothesis
({\bf G}) and the  $\mu_0$-invariant half-space
  is  $W+w_0+\R^+w=\{u: \pi_w(u)\in \R^+w \}$. This implies that $\pi_w(B_1)\in \R^+ w $ almost surely.
  Moreover,
$$
\int_G \log^-\big|\pi_w(b)\big| \mu_0(db\,da) = \int_G \log^- \big|\pi_w(b  + (1- a)w_0)\big|\mu(db\,da)<\8.
$$
Therefore, $\mu_0^w$ satisfies  the hypothesis of step 1 and
$\nu_0^w$ satisfies \eqref{ineq-nu-dx}. Hence for any nonnegative
and nonincreasing function $f$ on
  $\R$:
\begin{align*}
\int_{\R^d}f(|u|)\nu(du)&\leq \int_{\R^d}f(|\pi_w(u)|)\nu(du)\leq \int_{\R^d}f(|\pi_w(u+w_0)|-|w_0|)\nu(du)\\
&=\int_{\R}f(|x|-|w_0|)\nu_0^w(dx)\leq C' \left(\|f\|_\8+\int_{1/e}^\8 f(y-|w_0|)\frac{dy}{y}\right)\\
&= C' \left(\|f\|_\8+\int_{1/e}^{1/e+|w_0|}f(y-|w_0|)\frac{dy}{y}+\int_{1/e+|w_0|}^\8 f(y-|w_0|)\frac{y-|w_0|}{y}\frac{dy}{y-|w_0|}\right)\\
&\leq C\left(\|f\|_\8+\int_{1/e}^\8 f(y)\frac{dy}{y}\right)
\end{align*}

To prove  \eqref{nulog} we set $f=\Ind{(-\infty,z]}$. Then
$$\nu\{|u|\leq z\}\leq C(2+ \log z).$$

\end{proof}
%\begin{cor}
%\label{5.12}
%There exist constants $C_1$ and $C_2$ (they might be zero) such that
%\begin{eqnarray*}
%\liminf_{x\to\8} \nu(\a x,\b x)&=& C_1,\\
%\limsup_{x\to\8} \frac{\nu(\a x,\b x)}{\log x}&=& C_2.
%\end{eqnarray*}
%\end{cor}
\section{Recurrent potential kernel and solutions of the Poisson equation for general probability measures}
\label{sect-poisson} As it has been observed in the introduction,
 to understand  the asymptotic behavior of the measure $\nu$ one
has to consider of the function
$$f_{\phi}(x) = \int_{\R^d}\phi(u\e^{-x})\nu(du)$$
%as $x$ tedns to infinity,
that is a solution of the Poisson equation
\begin{equation}
\label{poisson-equation} \ov\mu *_\R f = f+\psi
\end{equation}
for a peculiar choice of the function $\psi$, that is
$$\psi_\phi = \ov\mu*_\R f_{\phi}-f_{\phi}.$$

Studying solutions of such equation for a centered probability
measure $\ov \mu$ on $\R$ is a classical problem. Port and Stones
in their papers \cite{PS1} and \cite{PS2} give an explicit formula
describing all bounded from below solutions of
\eqref{poisson-equation} in term of the recurrent potential kernel
$A$ of the function $\psi$. However, to obtain this result,
they suppose either that the measure is spread-out or, if not,
that functions $\psi $ satisfy conditions to restrictive from our
point of view.  Therefore, the
previous results on the decay of the measure $\nu$
  were obtained in \cite{Bu} under the hypothesis that $\ov \mu$ is spread out.
 The goal of this section is to generalize Port and Stone technics to arbitrary measures,
 that do not satisfy any smoothness conditions,
and to an appropriate class of functions $\psi$ depending on the
measure $\ov \mu$.

\medskip
Let $\ov \mu$ be a centered probability measure on $\R$ with
second moment $\s^2 = \int_\R x^2\ov\mu(dx)$ (we do not assume in
this section that $\ov\mu$ is related to $\mu$). If we exclude the
degenerate case when $\ov\mu=\d_0$, the closed group $G(\ov\mu)$
generated by the support of $\ov \mu$ can be either a discrete
group of the type $p\Z$ or $\R$. In the latter case, the measure
$\ov\mu$ is said aperiodic and we set $p=0$.

The Fourier transform of $\ov \mu$
$$\wh{\ov\mu}(\th) = \int_\R \e^{ix\th}\ov\mu(dx)$$
is a continuous bounded function (of period $2\pi/p$, if $\ov\mu$ is periodic), whose Taylor expansion near zero is
$$ \wh{\ov\mu}(\th) = 1+O(\th^2) $$
and such that
$$|1-\wh{\ov\mu}(\th)|>0 \quad \forall \th\in(0,2\pi/p)$$

Consider the set
 $\mathcal{F}(\ov\mu)$  of  functions $\psi$ that can be written as
$$ \psi(x) = \frac 1 {2\pi} \int_\R \e^{-ix\th} \wh{\psi}(\th) d\th$$
for some bounded, integrable, complex valued   function $\wh{\psi}$ verifying the following hypothesis
\begin{itemize}
  \item  its Taylor expansion  near $0$ is
  $$\wh\psi(\th) =   J(\psi)+i\th K(\psi) + O(\th^2)$$ for two constants $J(\psi)$ and $K(\psi)$,
  \item the function $\theta\mapsto \frac{\wh \psi(-\th)}{1-\widehat{\ov\mu}(\theta)}\cdot {\bf 1}_{[-a,a]^c}(\theta)$
is integrable for some $a\in(0,2\pi/p)$.
\end{itemize}

Notice that the first  condition is satisfied  when $\psi$ is a
continuous integrable  function, such that $x^2 \psi$ is
integrable and  whose Fourier transform is integrable. In this
case:
$$ J(\psi)=\int_\R \psi(x)dx\quad\mbox{and}\quad K(\psi)=\int_\R x\psi(x)dx.$$
The second condition is satisfied when  the measure is aperiodic
and  the Fourier transform of $\psi$ has compact support or in the
case the measure $\ov\mu$ is spread-out (since is this case
$\limsup_{|\theta|\to \infty}|\widehat{\ov\mu}(\theta)|<1$). Thus,
the set $\mathcal{F}(\ov\mu)$ contains the set of functions on
which Port and Stone define the recurrent potential and, in many
cases, it is bigger. We will see that even in the periodic case
$\mathcal{F}(\ov\mu)$ contains interesting functions and if we
suppose that $\ov \mu$ is $p$-nonlattice (i.e.
$\liminf_{\th\to\infty}|\th|^p|1-\wh{\ov\mu}(\th)|>0$, see \cite{C,B}), then $\mathcal{F}(\ov\mu)$ contains the Schwartz space.

%We will see that
%Also observe that, since $\wh{\ov\mu}$ is bounded, the two conditions implie that $\wh{\psi}$ is in $L^1(\R)$. As $\psi$ is a multiple of its fourier transform then $\psi$ has to be continuous and vanish at infinity.

\medskip

For $0< \l<1$ let
$$G^\l*\psi = \sum_{n=0}^\8\l^n\ov\mu^{*n}*\psi,$$
where $\ov\mu^{*n}$ denotes the $n$-th convolution power of
$\ov\mu$. One can easily see that the foregoing series is
convergent when $\psi$ is a bounded  measurable function. Next we
define
$$
A^\l \psi = c_\l J(\psi) - G^\l *\psi,
$$
where  $c_\l=G^\l *g(0)$ for some fixed positive function
  $g$ in $\mathcal{F}(\ov\mu)$ such that $J(g)=1$.

We are going to generalize the classical results of Port and Stone
to functions $\psi\in\mathcal{F}(\ov\mu)$ and to show that then
the limit value of $A^\l\psi$ exists and provides solutions of the
Poisson equation \eqref{poisson-equation}. We state here the main
results that are proved in Appendix \ref{port-stone}.

\medskip
\begin{thm}
\label{ps}
Assume that $\psi\in\mathcal{F}(\ov\mu)$.
Then the potential
$$A\psi (x)=
\lim_{\l\nearrow 1}A^\l\psi(x)$$ is a well defined continuous  solution
 of the Poisson equation \eqref{poisson-equation}.
 Furthermore if $J(\psi)\geq 0$ then $A\psi$ is bounded from below and
 \begin{equation}\label{eq-lim-A/x}
  \lim_{x\to\pm\8} \frac{A\psi(x)}{x} = \pm\s^{-2} J(\psi).
\end{equation}
If additionally $J(\psi)=0$, then $A\psi$ is  bounded and has a limit at infinity
\begin{equation}\label{eq-limA}
 \lim_{x\to\pm\8} A\psi(x) = \mp\s^{-2} K(\psi).
\end{equation}
\end{thm}
\begin{cor}
\label{ps1}
If  $J(\psi)=0$, then every continuous solution of the Poisson equation bounded from below is
 of the form $$f=A\psi+h$$ where $h$ is constant if $\ov\mu$ is aperiodic, and it is
 periodic of period $p$ if the support of $\ov\mu$ is contained in $p\Z$.
Thus every continuous solution of the Poisson equation is bounded
and the limit of $f(x)$ exists when $x$ goes to $+\infty$ and
$x\in G(\ov\mu)$.

Conversely if there exists a bounded solution of the Poisson
equation, then $A\psi$ is bounded and $J(\psi)=0$. In particular
the first part of corollary is valid.
\end{cor}

Using results of Baldi \cite{Ba}  it is possible to give an
explicit decomposition of the solutions of the Poisson equation,
also in the case $J(\psi)\neq 0$. The following result, although
not needed in the sequel, is stated  for completeness.
\begin{cor}\label{cor-form-solution}
For every continuous solution $f$ of the Poisson equation
\eqref{poisson-equation} bounded from below there are two
constants $C_1$ and $C_2$ such that
$$f(x)=A\psi(x)+C_1 J(\psi) x+ C_2 \quad \mbox{ for all $x\in G(\ov\mu)$.}$$
\end{cor}

The next lemma describes a class of functions in
$\mathcal{F}(\ov\mu)$ that we will be used later on and that have
the same type of decay at infinity as $\ov\mu$. In particular we
see that if $\ov\mu$ has exponential moment then
$\mathcal{F}(\ov\mu)$ contains functions with exponential decay.
\begin{lem}
\label{phi0} Let $Y$ a random variable with the law $\ov\mu$, then
the function
$$r(x)=\esp{|Y-x|-|x|}$$
  is nonnegative  and
\begin{equation*}
%\label{7.1}
\wh r (\th) = C \cdot  \frac{\wh{\ov \mu}(\th) - 1}{\th^2}
\end{equation*}
 for $\th\not= 0$. Moreover
 \begin{itemize}
\item[\refstepcounter{equation}(\theequation)\label{as1}]
if $\E[\e^{\d Y}+\e^{-\d Y}]<\8$, then
 $ r (x)\le C \e^{-\d_1|x|}$ for $\d_1 < \d$;
\item[\refstepcounter{equation}(\theequation)\label{as2}]
if $\E|Y|^{4+\eps}<\8$ for some $\eps>0$ then $ r (x)\le \frac C{1+|x|^{3+\eps}}$.
\end{itemize}
Hence if \eqref{as2} holds, $r$ is in $\mathcal{F}(\ov\mu)$ and
for every function $\zeta\in L^1(\R)$ such that $x^2\zeta$ is
integrable the convolution $r*_\R\zeta$ is in
$\mathcal{F}(\ov\mu)$.
\end{lem}

\section{Proofs of Theorems \ref{mthm1} and \ref{mthm2} - existence of the limit}
\label{sect-existence} First we are going to prove the following
result that holds for generic $\ov \mu$ not necessarily aperiodic.
\begin{prop}\label{mprop} Suppose that hypothesis {\bf (H)} is satisfied and either  (\ref{h1}) or (\ref{h1'}) holds.
Then the family of measures $\d_{(0,\e^{-x})}*_G\nu$ is relatively compact in
 the weak topology on $\R^d\setminus \{0\}$ and, when $x$ goes to infinity,
 every limit measure $\eta$ is invariant by the action of $G(\mu_A)$ that is $$\d_{(0,a)}*\eta=\eta \quad \forall a\in G(\mu_A).$$
  Furthermore for any function of the type
  \begin{equation}\label{eq-phispec-def}
{\phi}(u)=\int_\R r (t)\zeta(\e^{t}u)dt ,
  \end{equation}
  where
  \begin{equation}
  \label{def-r}
  r(t)=\esp{|-\log A_1-t|-|t|}
  \end{equation}
and    $\zeta$ is a nonnegative Lipschitz function on $\R^d\setminus\{0\}$ such that $\zeta(u)\leq \e^{-\g|\log|u||}$  for some $\g>0$,
 the limit
$$\lim_{x\to +\8}\int_{\R^d}\phi(u\e^{-x})\nu(du)=:T(\phi)$$

   exists, it is finite and equal to $\eta(\phi)$ for any limit
measure $\eta$.
\end{prop}

We have already observed that the function $f_{\phi}(x)=
\int_{\R^d}\phi(u\e^{-x})\nu(du)$ is a solution of the Poisson
equation associated to the function $\psi_\phi$. Our aim is to
apply the results of section \ref{sect-poisson}. For we need to
show that $\psi_\phi$ is sufficiently integrable. The upper bound
of the tail of $\nu$ given in section \ref{sect-upper-bound} will
guarantee integrability for positive $x$. To control the function
for $x$ negative we need to perturb slightly the  measures $\mu$
and $\nu$ in order to have more integrability near 0. This is
included in the following lemma
\begin{lemma} \label{lem-nu0}
 For all $x_0\in\R^d$ the translated  measure $\nu_0 = \delta_{x_0}*_{\R^d}\nu$
 is the unique invariant measure for $\mu_0 = \delta_{(x_0,1)}*_G \mu *_{G}\delta_{(-x_0,1)}$
  and it has the same behavior as $\nu$ at infinity, that is:
\begin{equation*}
\lim_{x\to +\8}\bigg(\int_{\R^d}\phi(u\e^{-x})\nu(du)-\int_{\R^d}\phi(u\e^{-x})\nu_0(du) \bigg) = 0
\end{equation*}
for every function $\phi\in C^1_c(\R^d\setminus\{0\})$

Furthermore  there is $x_0\in\R^d$ such that:
\begin{itemize}
\item if $\mu$ satisfies (\ref{h1}) then the same holds for $\mu_0$ and the measure $\nu_0$  satisfies
\begin{equation}
\label{nu0}
\int_{\R^d} \frac 1{|u|^\g}\nu_0(du)<\8 \mbox{ for all  $\g\in(0,1)$ }
\end{equation}

\item if $\mu$ satisfies (\ref{h1'}) then
\begin{equation}\label{nu0'}
 \int_G (|\log a| + \log^+|b|)^{4+\eps} \mu_0(db\,da) < \8,
 \mbox{ supp}\nu_0\subset W_0+\R^+ w \mbox{ and dist}(0;W_0+\R^+ w)>2
\end{equation}
where $W_0$ is an affine subspace of dimension $d-1$ orthogonal to the unit vector $w$
\end{itemize}

\end{lemma}

\begin{proof}
 If $\phi\in C^1_c(\R^d\setminus\{0\})$, by the Lipschitz property there exists a compact set  $K=K(x_0)$  such that
 if $|x_0 \e^{-x}|<1$, then
 $$\Big| \phi(u \e^{-x})- \phi((u+x_0) \e^{-x})\Big|\leq C \e^{-x} |x_0| {\bf 1}_K(u \e^{-x})$$
 for every  $u\in\R^d$.
Hence by Theorem \ref{thm-upp-bound},
 \begin{equation}
 \label{nu0-nu}
 \begin{split}
 \lim_{x\to +\infty} \bigg| \int_{\R^d} \phi(u \e^{-x})\nu(du)- \int_{\R^d} \phi(u \e^{-x})\nu_0(du) \bigg|
\le&\, C \lim_{x\to +\infty} \e^{-x} |x_0| \int_{\R^d}{\bf 1}_K(u \e^{-x})\nu(du)
  \\
 \le&\,  C |x_0|\lim_{x\to \infty} \e^{-x}L(\e^x)=0
 \end{split}
 \end{equation}

It is  easy to check that the integrability at infinity of $\mu$
and $\mu_0$ are the same since
$$\int_Gf(a,b+(1-a)x_0)\mu(db\,da)=\int_Gf(a,b)\mu_0(db\,da).$$
Notice also that the projections of $\mu$ and $\mu_0$ onto the
$A$-part coincide i.e. $\mu_A = \pi_A(\mu_0)$, in particular the
measure $\ov \mu$ defined by \eqref{overmu} is the same for both
$\mu$ and $\mu_0$.

In the case (\ref{h1'}) the support of $\nu$ is contained in some
half-space $W+\R^+ w$. Therefore

the support of $\nu_0$ is contained in $W+x_0+\R^+ w$. Let
$W_0=W+x_0$, we may choose $x_0$  in such a way that
dist$(0,W_0+\R^+w)>2$.
\medskip

In the case (\ref{h1}), since by Theorem \ref{thm-upp-bound}
$\d_{(0,\e^{-x})}*\nu_0(K)$ is smaller than a slowly varying
function as~$\nu$, then $|u|^{-\g}$ is $\nu_0$-integrable on
$|u|\ge 1$.

 For integrability near zero observe that
 $$\int_{\R^d}\int_{\R^d} \frac 1{|u+x|(\log|u+x|)^2}{ \bf 1}_{|u+x|<1}{ \bf 1}_{|x|<1}\nu(du)dx\leq\int_{\R^d}\int_{\R^d} \frac 1{|x|(\log|x|)^2}{ \bf 1}_{|x|<1}{ \bf 1}_{|u|<2}\nu(du)dx<\infty.$$
 Then there exists $x_0$ such that
 $$\int_{\R^d}  \frac 1{|u|(\log|u|)^2}{ \bf 1}_{|u|<1}\nu_0(du)=\int_{\R^d}\frac 1{|u+x_0|(\log|u+x_0|)^2}{ \bf 1}_{|u+x_0|<1}\nu(du)<\infty .$$
 Since for $|u|<1$ the function  $|u|^{-\gamma}$ is bounded by a multiple of $|u|^{-1}(\log|u|)^{-2}$,
 \eqref{nu0} follows.
\end{proof}

\medskip

If we suppose that (\ref{nu0}) (or (\ref{nu0'})) hold for $\nu$, we can guarantee that the function
  $\psi_\phi$ decays quickly at infinity, as it is proved in the  two following lemmas, in the first
  one for  the generic case and the second under
  the hypothesis {\bf (G)}.
\begin{lemma}
\label{decayofpsi}
Assume that \eqref{h1} is satisfied and that the function $|u|^{-\gamma}$ is $\nu(du)$ integrable for all $\gamma\in (0,1)$.  Let $\phi$ be a continuous

function on $\R^d$ such that
$$  |\phi(u)|\le \frac C{(1+|u|)^{\b}}$$ for some $\b , C>0$.
Then  $f_\phi$ and $\ov\mu * f_{\phi}$ are well defined and continuous.

Furthermore if $\phi$ is Lipschitz, then
\begin{equation}
\label{int}
\int_\R\int_{G}\int_{\R^d}\Big| \phi(\e^{-x}(au+b)) - \phi(\e^{-x}au) \Big|\nu(du) \mu(db\,da)dx < \8
\end{equation}
and
\begin{equation*}
%\label{wr4}
|\psi_\phi(x)| \le C \e^{-\zeta|x|},
\end{equation*}
for $\zeta<min\{\d/4,\b,1\}$.
\end{lemma}

\begin{proof}
If $\zeta < \min\{\beta,1\}$, then
$$|f_\phi(x)|\leq \int_{\R^d}\left|\phi(\e^{-x}u)\right|\nu(du) \le\int_{\R^d}\frac C{\e^{-\zeta x}|u|^{\zeta}}\nu(du) \le C \e^{\zeta x}.$$
If we suppose also $\zeta\leq \delta$, we have  that
%\begin{eqnarray*}
$$|\ov\mu*f_\phi(x)|\leq \int_{\R} \big|f_\phi (x+y) \big|\ov\mu(dy)
\le C \e^{\zeta x }\int_{\R^+} a^{-\zeta} \mu_A(da) \le C \e^{\zeta x}.
$$
Thus $\psi_\phi=\ov\mu*_\R f_\phi-f_\phi$ is well defined, continuous and
$$ |\psi_\phi(x)|\le C{\e^{\zeta x}}, $$
that gives the required estimates for negative $x$. In order to
prove \eqref{int} we divide the integral into two parts. For
negative $x$ we use the estimates given above:
\begin{multline*}
\int_{-\infty}^0\int_{G}\int_{\R^d}\Big| \phi(\e^{-x}(au+b)) - \phi(\e^{-x}au) \Big|\nu(du) \mu(db\,da)dx \\
\leq \int_{-\infty}^0\int_{G}\int_{\R^d}\big| \phi(\e^{-x}(au+b)) \big|\nu(du) \mu(db\,da)dx
+ \int_{-\infty}^0\int_{G}\int_{\R^d}\big|\phi(\e^{-x}au) \big|\nu(du) \mu(db\,da)dx\\
\leq \int_{-\infty}^0\int_{\R^d}\big| \phi(\e^{-x}u) \big|\nu(du)dx+
\int_{-\infty}^0\int_{G}\int_{\R^d}\big|\phi(\e^{-x}au) \big|\nu(du) \mu(db\,da)dx\\
 \le   \int_{-\8}^0  \big| f_{|\phi|}(x)\big|dx + \int_{-\8}^0 \big|\ov\mu*f_{|\phi|}(x)\big|dx
< \8.
\end{multline*}
To estimate the integral of $|\phi(\e^{-x}au) -
\phi(\e^{-x}(au+b))|$ for $x$ positive, we use
 the Lipschitz property of $\phi$ to  obtain the  following inequality for $0\leq\theta\leq 1$
\begin{eqnarray*}
|\phi(s) - \phi(r)| &\le&  2|\phi(s) - \phi(r)|^{\th}\max\Big\{
|\phi(s)|^{1-\th}, |\phi(r)|^{1-\th} \Big\}\\
&\le& C|s-r|^{\th}\max_{\xi\in \{ |s|,|r|\}}\frac
1{(1+\xi)^{\beta(1-\th)}}.
\end{eqnarray*}
Again we divide the  integral  into two parts. First we consider
the integral over the set where  $|a u+b| \ge \frac 12 a |u|$. We
choose $\th<\min\{ \d/2 ,1\}$, $\g<\min\{ \th/2,\b(1-\th) \}$.
Then
\begin{multline*}
  \int\int_{|a u+b| \ge \frac 12 |au|} \left|\phi(\e^{-x}au) - \phi(\e^{-x}(a u+b))\right|\mu(db\,da)\nu(du)\\
  \leq \int_G\int_{\R^d} \frac {C|\e^{-x}b|^\theta }{(1+| \e^{-x} au|)^{\beta(1-\th)}}\nu(du)\mu(db\,da)
  \leq \int_G\int_{\R^d}  \frac {C|\e^{-x}b|^\theta }{|\e^{-x} au|^{\g}}\nu(du)\mu(db\,da)\\
  \leq C \e^{-(\theta-\gamma)x}\int_G |b|^\theta |a|^{-\gamma}\mu(db\,da)\ \int_{\R^d} |u|^{-\g}\nu(du)\\
  \leq C \e^{-(\theta-\gamma)x } \int_G\big( |b|^{2\theta} + |a|^{-2\gamma}\big) \mu(db\,da)
  \leq C \e^{-\g x}.
\end{multline*}
If   $|a u+b|< \frac 12 a|u|$ then $|u| \le \frac{2 |b|}a $.
Therefore choosing $\th$ as above and $\g<\frac\d2-\th$, in view
of Theorem \ref{thm-upp-bound}, for the remaining part we have
\begin{multline*}
  \int\int_{|a u+b| \le \frac 12 |au|} \left|\phi(\e^{-x}au) - \phi(\e^{-x}(a u+b))\right|\mu(db\,da)\nu(du)\\
  \leq \int \int_{|u| \le \frac{2 |b|}a}|\e^{-x}b|^\theta  \nu(du)\mu(db\,da)
  \leq C\int_G |\e^{-x}b|^\theta \Big(1+\frac{2 |b|}a\Big)^\g \mu(db\,da) \\
  \leq C \e^{-\theta x}\int_G |b|^\theta \Big(1+\frac{2 |b|}a\Big)^\g\mu(db\,da)\\
  \leq C \e^{-\theta x} \int_G \big( |b|^\theta +|b|^{2(\th+\g)}+a^{-2\g}\big)\mu(db\,da)
  \leq C \e^{-\theta x}.
\end{multline*}
That proves \eqref{int} and
finally
\begin{eqnarray*}|\psi_\phi(x)|&=&\left|\int_G\int_{\R^d}\phi(\e^{-x}au)\nu(du)-\phi(\e^{-x}(au+b))\nu(du)\mu(db\,da)\right|\\
&\leq&\int_G\int_{\R^d}\big|\phi(\e^{-x}au)\nu(du)-\phi(\e^{-x}(au+b))\big|\nu(du)\mu(db\,da)< C \e^{-\zeta|x|}.
\end{eqnarray*}
for $\zeta<\min\{\d/4,\b,1\}$
\end{proof}

%%%%%%%%%%%%%%%%%%%%%%%%%%%%%%%%%%%%%% B Positive
\begin{lem}
\label{7.10}
If \eqref{nu0'} is satisfied and $\phi$ is a continuous function on $\R^d$ such that for $\beta >2$
$$ |\phi(u)| \le\frac C{(1+\log^+|u|)^\beta}, $$
then the functions $f_{\phi}$ and $\ov\mu *f_{\phi}$ are well defined.
Furthermore if $\phi$ is Lipschitz and $\b>4$, then
\begin{equation}
\label{int2}
\int_\R\int_{G}\int_{\R^d}\Big| \phi(\e^{-x}(au+b)) - \phi(\e^{-x}au) \Big|\nu(du) \mu(db\,da)dx < \8.
\end{equation}
and
$$
|\psi_\phi(x)| \le \frac C{1+|x|^{\chi}},
$$ for $\chi = \min\{ \b-1,3+\eps \}$.
\end{lem}
\begin{proof}
Assume first $x<-1$. In view of Proposition \ref{prop2} we have
\begin{eqnarray*}
\big|f_{\phi}(x)\big| &=& \int_{|u|>2}\big|\phi(\e^{-x }u)\big|\nu(du) \le
\int_{|u|>2} \frac{C}{\log^\b(\e^{-x}|u|)} \nu(du)\\
&\le&C\sum_{n=0}^\8 \int_{\e^n\le |u| < \e^{n+1}}\frac 1{(n-x)^\b}\nu(du)\\
&\le&  C\sum_{n>|x|}^\8\frac 1{n^\b} \int_{\e^{n+x}\le |u| < \e^{n+x+1}}\nu(du)\\
&\le&  C\sum_{m=1}^\8 \sum_{m|x|\le n <(m+1)|x|} \frac 1{m^\b|x|^\b} \int_{\e^{n+x}\le |u| < \e^{n+x+1}}\nu(du)\\
&\le&  C\sum_{m=1}^\8  \frac 1{m^\b|x|^\b} \int_{ |u| < \e^{(m+1)|x|}}\nu(du)
\le \frac{C}{|x|^{\b-1}}\sum_{m=1}^\8 \frac 1{m^{\b-1}} \\
&\le& \frac C{|x|^{\b-1}}.
\end{eqnarray*}
To proceed with positive $x$ notice that, by Proposition
\ref{prop2}, for every $y\in \R^+$ and $\b'>2$, arguing as above,
we obtain:
\begin{equation}
\label{est}
\begin{split}
\int_{\R^d}\frac 1{1+\big( \log^+(y|u|) \big)^{\b'}}\nu(du)
&\le \int_{y|u|<1}\nu(du) + \sum_{n=0}^\8 \int_{\e^n\le y|u|< \e^{n+1}} \frac 1{1+ n^{\b'}}\nu(du)\\
&\le C+C|\log y| + C\sum_{n=1}^\8 \frac 1{1+ n^{\b'-1}} \le C(1+|\log y|)
\end{split}
\end{equation}
Hence $|f_{\phi}(x)|\le C (1+x)$ if $x>0$.

\medskip

Finally $f_{\phi}$ is continuous, hence for $x\in(-1,0)$ is bounded. Thus
$$|f_\phi(x)| \leq C  \left((1+|x|){\bf 1}_{x > 0} + \frac 1{1+|x|^{\beta-1}}{\bf 1}_{x \le 0}\right)$$
Consider now the convolution of $f_{\phi}$ with $\ov\mu$. First if $x>0$, then
$$
\big| \ov\mu * f_{\phi}(x) \big| \le C \int_\R \big( 1+|x+y| \big)\ov\mu(dy) \le C\big(1+|x|\big).
$$
Next if $x<-1$, then since $\E|\log A|^{4+\eps}<\8$, we have
\begin{eqnarray*}
\big| \ov\mu * f_{\phi}(x) \big| &\le& \int_\R\frac C{1+|x+y|^{\b-1}}\ov\mu(dy)\\
&\le& \int_{2|y|<|x|}  \frac C{1+|x+y|^{\b-1}}\ov\mu(dy) + \frac C{|x|^{4+\eps}} \int_{2|y|\ge|x|} |y|^{4+\eps}\ov\mu(dy)\\
&\le& \frac C{1+|x|^{\chi_0}},
\end{eqnarray*}
 for $\chi_0=\min\{ \b-1,4+\eps \}$.
The function $\ov\mu * f_{\phi}$ is also continuous, hence finally
we obtain
%
%Let $\chi= 4+\eps$, since $\E{|\log A|^\chi}<\infty$
%\begin{eqnarray*}
%\int (1+|x+y|){\bf 1}_{x+y > 0}\,\ov\mu(dy)&\le& C(1+|x|){\bf 1}_{x > 0}+ {\bf 1}_{x \le 0} \int_\R  (1+ |x|)\frac{1+|y|^\chi}{1+|x|^\chi}+ |y|\frac{1+|y|^{\chi-1}}{1+|x|^{\chi-1}}\ov\mu(dy)\\
%&\le& C \left((1+|x|){\bf 1}_{x > 0} + \frac 1{1+|x|^{\chi-1}}{\bf 1}_{x \le 0}\right).
%\end{eqnarray*}%
%
%Furthermore if $\gamma=\min\{\beta-1,\chi-1\}$
%\begin{eqnarray*}
%  \int_\R \frac 1{1+|x+y|^{\beta-1}}\ov\mu(dy)&=&\frac 1{1+|x|^\gamma}\int_\R \frac {1+|x|^\gamma}{1+|x+y|^{\beta-1}}\ov\mu(dy)\\
%  &\leq& C\frac 1{1+|x|^\gamma}\int_\R \frac {1+|x+y|^\gamma+|y|^\gamma}{1+|x+y|^{\beta-1}}\ov\mu(dy)\\
%  &\leq& C\frac 1{1+|x|^\gamma}\int_\R 1+|y|^\gamma\ov\mu(dy)= C\frac 1{1+x^\gamma}
%\end{eqnarray*}
%then
%\begin{eqnarray*}

$$|\ov\mu*f_\phi(x)|
%&\leq&\int C (1+|x+y|){\bf 1}_{x+y > 0}+ \frac 1{1+(x+y)^{\beta-1}}{\bf 1}_{x+y\leq  0}\,\ov\mu(dy)\\
\leq C \left((1+|x|){\bf 1}_{x > 0} + \frac 1{1+|x|^{\chi_0}}{\bf 1}_{x \le 0}\right).
$$
%\end{eqnarray*}
Proceeding as in the previous lemma we prove
\begin{multline*}
\int_{-\8}^0\int_{G}\int_{\R^d}\Big| \phi(\e^{-x}(au+b)) - \phi(\e^{-x}au) \Big|\nu(du) \mu(db\,da)dx \\
\le \int_{-\8}^0 \big| f_{|\phi|}(x) \big| dx +  \int_{-\8}^0 \big| \ov\mu*f_{|\phi|}(x) \big| dx <\8
\end{multline*}

For $x>0$ we divide the integral of $\big| \phi(\e^{-x}au)- \phi(\e^{-x}(b+au)) \big|$ into several parts
and we  use the following inequality, being a consequence of the Lipschitz property of $\phi$:
$$ |\phi(s)-\phi(r)|\le C |s-r|^{\theta} \max_{\xi\in\{|s|,|r|\}}\frac 1{1+ (\log^+\xi)^{\beta'}}, $$
where $\theta < 1-2/\b$ and $\beta'=\beta(1-\theta)>2$.

\medskip

\noindent
{\bf Case 1.} First we assume  $|b|\le \e^{\frac x2}$.
Then by \eqref{est}
\begin{multline*}
\int_{ |b|\le \e^{\frac  x2} }\int_{\R^d} \Big| \phi(\e^{-x}au)- \phi(\e^{-x}(b+au)) \Big|\nu(du)\mu(db\,da)\\
\le C \int_{|b|\le \e^{\frac x2} }\int_{\R^d}
\e^{-\theta x}|b|^{\theta}\left(\frac 1 {1+ (\log^+(\e^{-x}a|u|))^{\beta'}}+\frac 1 {1+ (\log^+(\e^{-x}|au+b|))^{\beta'}}\right)\nu(du)\mu(db\,da)\\
\le C \e^{-\theta x/2}
\left(\int_{\Rp} \int_{\R^d} \frac 1 {1+ (\log^+(\e^{-x}a|u|))^{\beta'}}\nu(du) \mu_A(da) +\int_{\R^d} \frac 1 {1+ (\log^+(\e^{-x}|u|))^{\beta'}}\nu(du)\right)  \\
\le C \e^{-\theta x/2} \bigg[ 1+x +   \int_{\R^+} |\log a|  \mu_A(da)\bigg]  < C\e^{-\theta x/4}.
\end{multline*}
{\bf Case 2.} We assume $ a|u|<  2|au+b|$ and $|b| > \e^{\frac  x2}$.
Notice first
$$\int_{ |b| >  \e^{\frac  x2} }  \mu(db\,da)\le \frac C{1+x^{4+\eps}}\int_{\R^d}\Big( 1+\big(\log^+|b|\big)^{4+\eps}\Big)\mu(db\,da)
\le \frac C{1+x^{4+\eps}}.
$$
and
\begin{multline*}
\int_{ |b| >  \e^{\frac  x2} } \left(|\log a|+\log|b|\right)\mu(db\,da)\\
\le \frac C{1+x^{3+\eps}}\int_{G}(1+\big(|\log a|+\log^+|b|\big)^{3+\eps}(\log^+|b|+|\log a|)\mu(db\,da)
\le \frac C{1+x^{3+\eps}}.
\end{multline*}
Then, proceeding as previously, we have
\begin{multline*}
\int\int_{\substack{ a|u|<2|au+b| \\ |b|> \e^{\frac  x2} }}\Big| \phi(\e^{-x}au)- \phi(\e^{-x}(b+au)) \Big|\nu(du)\mu(db\,da)\\
\le 2\int\int_{\substack{  a|u|<2|au+b| \\ |b| >  \e^{\frac x2} }}
\max\Big\{\big|\phi(\e^{-x}au)\big|,\big| \phi(\e^{-x}(b+au))\big|\Big\} \nu(du)\mu(db\,da)\\
\le C\int_{ |b| >  \e^{\frac  x2} }\int_{\R^d}
\frac 1{1+ (\log^+(\e^{-x}a|u|))^{\beta}} \nu(du)\mu(db\,da)\\
\le C\int_{ |b| >  \e^{\frac  x2} } \big( x+|\log a| +1 \big)\mu(da\,db)\le \frac C{1+x^{3+\eps}}.
\end{multline*}

{\bf Case 3.}
The last case is  $ a|u|\ge 2|au+b|$ and $|b| > \e^{\frac  x2}$.
Then $|u|<\frac {2|b|}a$ and we obtain
\begin{multline*}
\int\int_{\substack{  a|u|\ge 2|au+b| \\ |b|> \e^{\frac  x2} }}\Big| \phi(\e^{-x}au)- \phi(\e^{-x}(b+au)) \Big|\nu(du)\mu(db\,da)\\
\le C\int_{|b| >  \e^{\frac  x2} }\int_{ |u|<\frac {2|b|}a }\nu(du)\mu(db\,da)\le C\int_{|b| >  \e^{\frac  x2}} \big(1+\log|b|+|\log a|\big)\mu(db\,da)
\le \frac C{1+x^{3+\eps}}.
\end{multline*}
We conclude \eqref{int2} and the required estimates for
$\psi_{\phi}$.
\end{proof}

%%%%%%%%%%%%%%%%%%%%%%%%%%%%%%%%%%%%%%%%%%%%%%%%%%%5

\begin{proof}[Proof of Proposition \ref{mprop}] {\bf Step 1.} First we suppose that $\mu$ satisfies either (\ref{h1}) and  (\ref{nu0}) or (\ref{nu0'}).

We are going to show that for functions of type
\eqref{eq-phispec-def} the limit
$$\lim_{x\to +\8}\int_{\R^d}\phi(u\e^{-x})\nu(du)=T(\phi):=-2\s^{-2}K(\psi_\phi)$$
exists and is finite. To do this we will prove that $\psi_\phi$ is an element of
$\mathcal{F}(\ov\mu)$ and $J(\psi_{\phi})=0$.
 Thus, by Corollary \ref{ps1},
  $f_\phi(x)=\int_{\R^d}\phi(u\e^{-x})\nu(du)$, that is the solution  of the corresponding Poisson equation,
  is bounded and it has a limit when $x$ converge to $+\infty$. We will prove that the limit exists even if values
  of $x$ are not restricted to $G(\ov\mu)$.

First observe that if $\esp{|A_1|^\d+|A_1|^{-\d}}<\8$, then by
lemma \ref{phi0}, for $\beta<\min\{\d,\g\}$,

we have
\begin{eqnarray*}%\label{eq-phispec}
  |{\phi}(u)|&\leq& C\int_\R \e^{-\beta |t|}\e^{-\g|t+\log|u||}dt\leq C\int_\R \e^{-\beta (|t-\log|u||)}\e^{-\g|t|}dt\\
&\leq&C\int_\R \e^{-\beta (-|t|+|\log|u||)}\e^{-\g|t|}dt = C \e^{-\beta|\log|u||}.
\end{eqnarray*}
In the same, way if $\esp{|\log A|^{4+\eps}}<\8$,

then
\begin{eqnarray*}%\label{eq-phispec}
  |{\phi}(u)|&\leq& C\int_\R \frac 1{1+|t-\log|u||^{3+\eps}}\e^{-\g|t|}dt\\
  &\leq&
  \frac C {1+|\log|u||^{3+\eps}}\int_\R\frac {1+|t-\log|u||^{3+\eps}+|t|^{3+\eps}}{1+|t-\log|u||^{3+\eps}}\e^{-\g|t|}dt\\
&\leq&\frac C {1+|\log|u||^{3+\eps}}\int_\R (1+|t|^{3+\eps})\e^{-\g|t|}dt\leq  \frac C {1+|\log|u||^{3+\eps}}.
\end{eqnarray*}
Thus by Lemmas \ref{decayofpsi} and \ref{7.10}, $f_\phi$,
$f_\zeta$,
 $\ov\mu*f_\phi$ and $\ov\mu*f_\zeta$ are well defined. Furthermore since $\zeta$ is Lipschitz $\psi_{\zeta}$ is bounded,
  and $x^2\psi_{\zeta}(x)$ is integrable on $\R$. We cannot guarantee that $\phi$ is Lipschitz, but we can
observe that
\begin{eqnarray*}
  f_{\phi}(x)&=& \int_{\R^d}\int_\RR r (t)\zeta(\e^{-x+t}u)dt\nu(du)\\
  &=& \int_{\R^d}\int_\RR r (t+x)\zeta(\e^{t}u)dt\,\nu(du)\\
  &=& \int_\RR r(t+x)f_\zeta(-t)dt\\
  &=& r *_\R f_\zeta(x)
\end{eqnarray*}
and
$$ \ov\mu*f_\phi(x)=r*_\R (\ov\mu*f_\zeta)(x).$$
Hence
$$\psi_{\phi}= f_\phi-\ov\mu*f_\phi=r*(f_\zeta-\ov\mu*f_\zeta)= r*_\R\psi_\zeta .$$
Therefore, by Lemma \ref{phi0}, $\psi_{\phi}\in {\mathcal
F}(\ov\mu)$.

Furthermore if $\zeta$ is radial then $J(\psi_{\phi})=0$. In fact,
let ${\zeta}_r$ be the radial part of ${\zeta}$, i.e.
${\zeta}_r(|u|) = {\zeta}(u)$, then
\begin{eqnarray*}
\int_\R \psi_{\zeta}(x)dx &=& \int_\R\int_G\int_{\R^d}\bigg[ {\zeta}(au \e^{-x})
-{\zeta}(\e^{-x}(au+b)) \bigg] \nu(du)\mu(db\,da) dx\\
&=&\int_G\int_{\R^d}\int_\R \bigg[ {\zeta}_r\Big( \e^{-x+\log(|au|)}\Big)
-{\zeta}_r\Big(\e^{-x+\log|au+b|}\Big) \bigg] dx\nu(du)\mu(db\,da)\\
&=&\int_G\int_{\R^d}\bigg(\int_\R {\zeta}_r( \e^{-x})dx-\int_\R {\zeta}_r(\e^{-x}) dx\bigg)\nu(du)\mu(db\,da)=0.
\end{eqnarray*}
Observe that we can apply the Fubini theorem since $\zeta$ is Lipschitz and, by Lemmas \ref{decayofpsi} and \ref{7.10},
 the absolute value of the integrand in the second line above is integrable.
Hence $$J(\psi_\phi)=\int_\R \psi_{\phi}(x)dx = \int_\R
r*\psi_{\zeta}(x)dx=\int_\R r(x)dx\cdot \int_\R \psi_{\zeta}(x)dx=
0.$$ If $\zeta$ is radial, then
 by Corollary \ref{ps1}, we have
\begin{equation}\label{eq-dec-fphi}
f_{\phi}=A\psi_\phi+h_\phi
\end{equation}
where $h_\phi$ is a constant if $\mu_A$ is aperiodic and a continuous periodic function if $\mu_A$ is periodic. In any case $f_\phi$ is bounded.

In particular the same holds for $f_{\Phi_\g}$, where
$${\Phi_\g}(u)=\int_\R r (t)\e^{-\g|t+\log|u||}dt.$$
For a generic non-radial function $\phi$ of the type
(\ref{eq-phispec-def}), there exists $\g>0$ such that $\phi\le
\Phi_\g$. Hence $f_{\phi}\le f_{\Phi_\g}$ and $f_{\phi}$ is a
bounded solution of the Poisson equation associated to
$\psi_{\phi}$. Therefore, by Corollary \ref{ps1}, $J(\psi_\phi)=0$
and
$$f_{\phi}=A\psi_\phi+h_\phi$$  for a periodic function $h_\phi$.

Since the measure $\nu$ has no mass in zero $\lim_{x\to -\infty}
f_\phi(x)=0$ and by Theorem  \ref{ps} $$ \lim_{x\to -\infty}
A\psi_\phi(x)=\s^{-2}K(\psi_\phi).$$ Thus when $x$ goes to
$-\infty$ the limit (not necessarily restricted to $G(\ov\mu)$) of
$h_\phi$ exists which is possible only if $h_\phi$ is constant and
is equal to $- \s^{-2}K(\psi_\phi)$. Finally
$$\lim_{x\to +\infty} f_\phi(x)=\lim_{x\to +\infty}A\psi_\phi(x)- \s^{-2}K(\psi_\phi)=-2\s^{-2}K(\psi_\phi).$$

\medskip

\noindent
{\bf Step 2.} Fix a $\g>0$.
Since $\Phi_\g>0$ for every function $\phi\in C_c(\R^d\setminus\{0\})$ there exists a constant $C_\phi$ such that $|\phi|\leq C_\phi\Phi_\g$.
Thus the family of  measures on $\RR^d\setminus\{0\}$
$$\delta_{(0,\e^{-x})}*_G\nu(\phi)=\int_{\R^d}\phi(\e^{-x}u)\nu(du)$$
is bounded hence it is relatively compact  in  the weak topology.

Let $\eta$ be an accumulation point for a subsequence $\{x_n\}$
that is
\begin{equation}
\label{xn_nu}
\lim_{n\to\infty}\d_{(0,\e^{-x_n})}*_G\nu(\phi)=\eta(\phi) \quad \forall \phi\in C_c(\R^d\setminus\{0\}).
\end{equation}
By Lemma \ref{lem-inv-lim} the measure $\eta$ is $G(\mu_A)$
invariant. We claim that for any continuous non negative function
such that $\phi\leq C_{\phi}\Phi_\g$, not necessarily compactly
supported,
$$\eta(\phi)=\lim_{n\to\infty}\delta_{(0,\e^{-x_n})}*_G\nu(\phi).$$
Indeed, fix a large constant $M$, take $\eps>0$ and define
$$
\phi_M(u) = \phi(u)\cdot \int_{\R}{\bf 1}_{[1/((1+\eps)M), M(1+\eps)]}(t|u|)h(t)dt,
$$ where $h\in C_C\big( (1+\eps)^{-1}, 1+\eps \big)$ and $\int_\R h(t)dt=1$.
Then $\phi_M$ is continuous, its support is contained in the annulus  $\{ u:\; 1/(M(1+\eps)^{2}) \le |u| < M(1+\eps)^2 \}$
and moreover $\phi_M(u)=\phi(u)$  for $1/M < |u| < M$.

Notice that by the monotone convergence theorem
$$\eta(\phi)= \lim_{M\to\infty}\int_{\R^d}\phi_M(u)\eta(du)$$
and by \eqref{xn_nu}, for a fixed $M$
$$
\lim_{n\to\infty}\delta_{(0,\e^{-x_n})}*_G\nu(\phi_M)=\eta(\phi_M).
$$
Therefore,
\begin{eqnarray*}
 |\eta(\phi)-\lim_{n\to\infty}\delta_{(0,\e^{-x_n})}*_G\nu(\phi)| &\leq&
 \lim_{M\to \8}\bigg[
  \big|\eta(\phi)-\eta(\phi_M)\big|+
  \Big|\eta(\phi_M) -  \lim_{n\to\infty}\delta_{(0,\e^{-x_n})}*_G\nu(\phi_M) \Big|\\
  && +\Big| \lim_{n\to\infty}\delta_{(0,\e^{-x_n})}*_G\nu(\phi - \phi_M)  \Big|\bigg]\\
  &\le& \lim_{M\to \8}  \lim_{n\to\infty}\delta_{(0,\e^{-x_n})}*_G\nu(|\phi - \phi_M|)
\end{eqnarray*}
and we have to prove that the last limit is 0. For that observe
that for every compact set $K$, there exists a constant $C_K$ such
that ${\bf 1}_K \le C_K\Phi_\g$, hence
$$
\d_{(0,\e^{-x})} *\nu(K) \le C_K',
$$
for every $x\in \R$. Now, applying Lemma \ref{lem-nu-leb} with $l(z)\equiv 1$, since
$\Phi_\g(u)\le \frac C{1+|\log|u||^{3+\eps}}$,

we obtain
\begin{multline*}
\lim_{n\to\8} \d_{(0,\e^{-x_n})} *\nu(|\phi-\phi_M|)
\le \sup_{x\in\R} \int_{\R^d} \Phi_\g(\e^{-x}u)\Big(1-{\bf 1}_{[1/M,M]}(\e^{-x}u)\Big)\nu(du)\\
\le \sup_{x\in\R} C\bigg( \int_{M/e}^\8 \frac 1{1+\big| \log(\e^{-x}a) \big|^{3+\eps}} \frac {da}a
+  \int_0^{e/M}  \frac 1{1+\big| \log(\e^{-x}a) \big|^{3+\eps}} \frac {da}a \bigg)\\
\le C \int_{\R^+} \frac 1{1+| \log a |^{3+\eps}}  \Big( {\bf 1}_{(0,e/M]} +  {\bf 1}_{[M/e,\8)} \Big) \frac {da}a
\end{multline*}
letting $M$ to go to infinity, we conclude.

\medskip

\noindent
{\bf Step 3.} Now we return to the general case of a measure $\mu$ satisfying (\ref{h1}) (or (\ref{h1'})). Then by lemma \ref{lem-nu0}
there exists $\nu_0=\d_{x_0}*_{\R^d}\nu$ for which (\ref{h1}) and (\ref{nu0}) (or (\ref{nu0'})) hold. We have  proved
in Lemma \ref{lem-nu0} that
 $\d_{(0,\e^{-x})}*_G\nu$ and $\d_{(0,\e^{-x})}*_G\nu_0$ have the same behavior on compactly supported functions when $x$ go to
 $+\infty$.
Thus we still need to prove that they have the same limit for
functions $\phi$ of the type (\ref{eq-phispec-def})
   even if they do not have compact support. Observe that both  sets of measures are bounded on compact set, and in particular:
$$  \sup_{x\in\RR}\d_{(0,\e^{-x})}*\nu(1\le|u|\leq e)=K<\infty\quad \sup_{x\in\RR}\d_{(0,\e^{-x})}*\nu_0(1\le|u|\leq e)=K_0<\infty.$$
By the Lipschitz property
$$\left|\zeta(\e^{t-x}u)-\zeta(\e^{t-x}(u+x_0))\right|\leq
 C  \e^{\th(t-x)} \left(\zeta(\e^{t-x}u)^{1-\theta}+\zeta(\e^{t-x}(u+x_0))^{1-\theta}\right)$$ for all $\theta\in
 [0,1]$.
 Hence by Lemma \ref{lem-nu-leb}
\begin{multline*}
  |\d_{(0,\e^{-x})}*\nu(\phi)-\d_{(0,\e^{-x})}*\nu_0(\phi)|
  \leq \int_\R \int_{\R^d} r(t)\left|\zeta(\e^{t-x}u)-\zeta(\e^{t-x}(u+x_0))\right|dt \nu(du)\\
  \leq C\bigg[ \int_x^\8 r(t)dt +
  \int_{-\8}^x r(t)\e^{\th(t-x)}\left(\int_{\R^d}\zeta(\e^{t-x}u)^{1-\theta}\nu(du)+\int_{\R^d} \zeta(\e^{t-x}u)^{1-\theta}\nu_0(du)\right)dt\bigg]\\
  \leq C \bigg[ \int_x^\8 r(t)dt +
  \int _{-\8}^x r(t)\e^{\th(t-x)} \left(\int_0^\8 \e^{-\g(1-\theta)|\log a|} \frac{da}a\right)dt\bigg]\\
  \leq C \bigg[ \int_x^\8 r(t)dt +
  \int_{-\8}^x  r(t)\e^{\th(t-x)} dt\bigg]\leq C \bigg[ \int_x^\8 r(t)dt +
  \int_{-\8}^x  r(t) dt\bigg]
\end{multline*}
Since, by the dominated convergence theorem, the last term goes to zero when $x$ go $+\infty$, we conclude.
\end{proof}
\begin{proof}[Proof of Theorem \ref{mthm1} - existence of the limit]
We assume that $\mu_A$ is aperiodic. Then in view of Proposition
\ref{mprop} the family of measures $\d_{(0,\e^{-x})}*_G \nu$ is
relatively compact  in the weak topology and if  $\eta$ is an
accumulation point, then it is $\Rp$ invariant i.e. for every
$a\in \Rp$
 $$ \int_{\R^d}\phi(au)\eta(du) =\int_{\R^d}\phi(u)\eta(du).$$
Therefore there exists a probability measure $\Sigma_\eta$ on $S^{d-1}$ and a constant $C_\eta$ such that $$\eta = C_\eta \Sigma_\eta \otimes \frac{da}a$$
 (see \cite{FS}, Proposition 1.15).
It remains to prove that $C_\eta$ and $\Sigma_\eta$ do not depend
on  $\eta$. We have proved in Proposition
\ref{mprop} that for  any function  $\phi$ of type
(\ref{eq-phispec-def}), the limit exists (that is
 it does not depend on the subsequence along which one tends to $\eta$)
$$\lim_{x\to +\8}\int_{\R^d}\phi(u\e^{-x})\nu(du)=\eta(\phi)=T(\phi).$$
Consider the radial function ${\Phi_\g}(u)=\int_\R r (t)\e^{-\g|t+\log|u||}dt$,
since $\eta(\Phi_\g)= C_\eta \int_{\R^*_+}\Phi_\g(a)\frac{da}a$. Then:
$$C_\eta=\frac{T({\Phi_\g})}{  \int_{\R^*_+}\Phi_\g(a)\frac{da}a}$$
does not depend on $\eta$. Set $C_+=C_\eta$.

For any Lipschitz function $\zeta_0$ of $S^{d-1}$ consider the function $\zeta(u)=\e^{-\g|\log|u||}\zeta_0(u/|u|)$ and
$${\phi}(u)=\int_\R r (t)\zeta(\e^{t}u)dt
= {\Phi_\g}(u)=\int_\R r (t)\e^{-\g|t+\log|u||}\zeta_0(e^{-t}u/|e^{-t}u|)dt=\Phi_\g(u)\zeta_0(u/|u|).$$
Then
$$\eta(\phi)=C_+\Sigma_\eta(\zeta_0)  \cdot \int_{\R^*_+}\Phi_\g(a)\frac{da}a=T(\phi)$$
thus $\Sigma_\eta(\zeta_0)$ does not depend on $\eta$.
\end{proof}

\begin{proof}[Proof of Theorem \ref{mthm2} - existence of the limit]
We proceed as in the previous proof.
Assume that $\mu_A$ is aperiodic and $G(\mu_A)=\langle e^p\rangle$. Let
$D=\{x:\; 1\le x<e^p\}$ be the fundamental domain for the action of $G(\mu_A)$ on $\R^d\setminus\{0\}$. Then every
$x\in \R^d\setminus\{0\}$ can be uniquely written as $x=aw$, where $a\in G(\mu_A)$
and $w\in D$. Denote by $l$ the counting measure on $G(\mu_A)$.

\medskip

First we will prove that the exists a radial function $\Phi$ of type \eqref{eq-phispec-def}
such that for every $a\in \R_+^*$, $\d_{(0,a)}*_G l(\Phi) = l(\Phi)$, where now
$l$ is considered as a measure on $\R^d\setminus\{0\}$, i.e. for any unit vector
$x\in \R^d$: $l(\Phi)=\sum_{k\in \Z}\Phi(e^{kp}x)$. For radial functions
$l(\Phi)$ does not depend on the choice of $x$.
Indeed, let $R\in C_C(\R^d\setminus\{0\})$ be the function defined in the proof
of Theorem \ref{thm-upp-bound}. Clearly, we may assume that $R$ is radial.
Then, since $\d_{(0,a)}*_Gl(R)=l(R)$ for all $a\in G(\mu_A)$ and repeating
the argument given at the end of the proof of Theorem \ref{thm-upp-bound},
we prove that $\d_{(0,a)}*_Gl(R)=l(R)$ for all $a\in \R^*_+$.
Take $\Phi(u)=\int_\R r(t)R(e^t u)dt$, then for every $a\in \R^*_+$ we have
%\begin{eqnarray*}
$$\d_{(0,a)}*_Gl(\Phi) = \int_\R r(t) \d_{(0,e^t)}*_G \d_{(0,a)}*_Gl(R)dt
 = \int_\R r(t) \d_{(0,e^t)}*_Gl(R)dt = l(\Phi).
$$%\end{eqnarray*}
 Let $\eta$ be an accumulation point of the family of measures $\d_{(0,e^{-x})}*_G\nu$.
 Then, in view of Proposition \ref{mprop}, $\eta$ is $G(\mu_A)$ invariant. Therefore there
 exists a probability measure $\Sigma_{\eta}$ on $D$ and a constant $C_{\eta}$ such that
 $\eta = C_{\eta}\Sigma_{\eta}\otimes l$. By Proposition \ref{mprop}, the value $T(\Phi)=\eta(\Phi)$
 does not depend on $\eta$. Notice that
 \begin{eqnarray*}
 T(\Phi) &=& C_{\eta} \int_D\int_{G(\mu_A)}\Phi(aw)l(da)\Sigma_{\eta}(dw)
 = C_{\eta} \int_D\d_{(0,|w|)}*_Gl(\Phi)\Sigma_{\eta}(dw)\\
 &=& C_{\eta} \int_Dl(\Phi)\Sigma_{\eta}(dw) = C_{\eta}l(\Phi).
 \end{eqnarray*}
 Therefore $C_{\eta}=\frac{T(\Phi)}{l(\Phi)}$ does not depend on $\eta$. Define
 $C_+=C_{\eta}$. Finally, we have
 \begin{eqnarray*}
 \eta\big( {\bf 1}_{C(z , e^{np}z)} \big) &=&
 C_+\int_D\int_{G(\mu_A)} {\bf 1}_{C(z , e^{np}z)} (aw) l(da) \Sigma_{\eta}(dw)\\
 &=&C_+ \int_D l(C(wz,wze^{np}) \Sigma_{\eta}(dw)= C_+n \int_D\Sigma_{\eta}(dw)\\
  &=& nC_+
 \end{eqnarray*} and this value also does not depend on $\eta$.
\end{proof}

\section{Positivity of the limiting constant}
\label{sect-positivity}

In this section we are going to discuss non degeneracy of the
limit measure \eqref{rlimit} and finish the proofs of theorems
\ref{mthm1} and \ref{mthm2} proving that the constant $C_+$ is positive.
 A partial result was obtained in
\cite{Bu} in the one dimensional case  and $B\ge \eps$ a.s in
\cite{Bu}. Now we are going to prove
\begin{thm}
\label{positive}
If hypothesis {\bf (H)} is satisfied,
then  for all $\a, \b>0$
 \begin{equation}\label{eq-positive}
\limsup_{x\to\8} \nu\big\{u\in\R^d:\;  z\alpha < |u|\le  z\beta  \big\} >0
.\end{equation}
\end{thm}
First we will prove a version of the duality lemma  generalized to
time-reversible functions. Although the technic of proof is
classical (see for instance \cite{F}), we present here complete
argument for reader's convenience. Let $W_i=(Y_i,Z_i)$ be a
sequence of i.i.d. random variables on $\R\times \R$ and let
$$S_n=\sum_{i=1}^nY_i \mbox{ if } n\geq 1 \quad\mbox{ and }
S_0=0$$ (later we will take $W_i=(\log A_i,B_i)$). We define a
sequence of stopping times:
\begin{eqnarray*}
T_0 &=& 0;\\
T_i &=& \inf\{ n>T_{i-1}:\; S_n > S_{T_{i-1}} \}
\end{eqnarray*}
and we put
$$ L= \inf \{n:\; S_n < 0\}. $$
If the events are void then the stopping times are equal to
$\infty$.
\begin{lem}[Duality Lemma]\label{lem-duality}\
Consider a sequence of non negative functions
$$\a_n: (\R\times \R)^n \rightarrow \R,$$
for $n\ge 1$, $\a_0$ equal to some constant and $\a_\8=0$, that are time reversible, that is
 $$ \a_n(w_1, \ldots, w_n)=\a_n(w_n, \ldots, w_1) \quad \forall (w_1, \ldots, w_n)\in (\R\times \R)^n.$$
 Then
$$
\E\bigg[  \sum_{i=0}^{L-1} \a_i(W_1,\ldots W_i) \bigg] = \E\bigg[ \sum_{i=0}^\8\a_{T_i}(W_1,\ldots W_{T_i}) \bigg].
$$
\end{lem}
\begin{proof}
We have
\begin{eqnarray*}
\E\bigg[  \sum_{i=0}^{L-1} \a_i(W_1,\ldots W_i) \bigg] &=&
\sum_{i=0}^\8\E\Big[{\bf 1}_{[i<L]} \a_i(W_1,\ldots W_i) \Big]\\
&=&\a_0+\sum_{i=1}^\8\E\Big[ {\bf 1}_{[S_j>0\ \forall j=1,\ldots,i]} \a_i(W_1,\ldots W_i)\Big]
\end{eqnarray*}

For fixed $i$ consider the reversed time sequence $$ \overline{W}_k=W_{i-k+1}$$ and observe that the vector $(\overline{W}_1,\ldots, \overline{W}_i)=(W_i, \ldots, W_1)$ has the same law as $(W_1,\cdots,W_i)$. Thus
\begin{eqnarray*}
  && \E\Big[ {\bf 1}_{[S_j>0\ \forall j=1,\ldots,i]} \a_i(W_1,\ldots W_i)\Big]=
  \E\Big[ {\bf 1}_{[\sum_{k=1}^jY_k>0\ \forall j=1,\ldots,i]} \a_i(W_1,\ldots W_i)\Big]\\
  &&=\E\Big[ {\bf 1}_{[\sum_{k=1}^j\overline{Y}_k>0\ \forall j=1,\ldots,i]} \a_i(\overline{W}_1,\ldots \overline{W}_i)\Big]=
  \E\Big[ {\bf 1}_{[\sum_{k=1}^jY_{i-k+1}>0\ \forall j=1,\ldots,i]} \a_i(W_i,\ldots W_1)\Big]\\
 &&=\E\Big[ {\bf 1}_{[\sum_{k=i-j+1}^{i}Y_k>0\ \forall j=1,\ldots,i]} \a_i(W_1,\ldots W_i)\Big]=
  \E\Big[ {\bf 1}_{[S_i-S_{i-j}>0\ \forall j=1,\ldots,i]} \a_i(W_1,\ldots W_i)\Big]\\
  &&=\E\Big[ {\bf 1}_{[S_i>S_j\ \forall j=0,\ldots,i-1]} \a_i(W_1,\ldots W_i)\Big]=\E\Big[ {\bf 1}_{[\exists k\geq 1 : i=T_k]} \a_i(W_1,\ldots W_i)\Big]
\end{eqnarray*}
Then
$$\E\bigg[  \sum_{i=0}^{L-1} \a_i(W_1,\ldots W_i) \bigg]=\a_0+\sum_{i=1}^\8\E\Big[ {\bf 1}_{[\exists k\geq 1 : i=T_k]}
\a_i(W_1,\ldots W_i)\Big]= E\bigg[ \sum_{k=0}^\8\a_{T_k}(W_1,\ldots W_{T_k}) \bigg].$$

\end{proof}

\begin{proof}[Proof of Theorem \ref{positive}]
{\bf Step 1.} First we claim that there exist two positive
constants $C$ and $M$ such that  for every positive nonincreasing
$f$ on $\R_+$
\begin{equation}\label{nu-leb}
\int_{\R^d}f(|u|)\nu(du)\geq C \int_M^\8 f(a)\frac{da}{a}.
\end{equation}

Take $Y_i=\log A_i$ and $S_n=\sum_{i=1}^nY_i$. As it was proved in
\cite{BBE} $$\nu(f) =  \int_{\R^d}\E\Big[ \sum_{n=0}^{L-1}f(X_n^x)
\Big]\nu_L(dx),$$ were $\nu_L$ is the invariant probability
measure of the process $X_{L_n}$. Take a ball $B$ of $\R^d$ of
radius $R$ such that $\nu_L(B)=C_R >0$.
    We have
\begin{eqnarray*}
\int_{\R^d}f(|u|)\nu(du) &\ge & \int_B\E\bigg[
\sum_{n=0}^{L-1} f\big( \big|A_1A_2\ldots A_n x+ A_2A_3\ldots A_n B_1+\cdots + B_n \big| \big)\bigg] \nu_L(dx)\\
&\ge& C_R  \E\bigg[
\sum_{n=0}^{L-1} f\Big( A_1A_2\ldots A_n\big(R+ |B_1|+\cdots+ |B_n| \big)\Big)\bigg] \\
&=&
C_R  \E\bigg[
\sum_{n=0}^{L-1} f\Big( \e^{S_{n}+\log(R+\sum_{i=1}^{n}|B_i|)}\Big)\bigg] \\
&=&  C_R  \E\bigg[
\sum_{n=0}^{\8} f\Big( \e^{S_{T_n}+\log(R+\sum_{i=1}^{T_n}|B_i|)}\Big)\bigg]
\end{eqnarray*}
In the last line we applied the duality Lemma since the function:
$$\alpha_n((Y_1,B_1),\ldots,(Y_n,B_n))=f\Big( \e^{\sum_{i=1}^{n}Y_i+\log(R+\sum_{i=1}^{n}|B_i|)}\Big)$$
is time reversible.
Consider the sequences of i.i.d. variables
$$W_j=  \max\{\log(1+R+|B_i|): i= T_{j-1}+1,\ldots, T_{j}\}$$ and
 $$V_j=S_{T_j}- S_{T_{j-1}}+\log (T_{j}-T_{j-1}) + W_j.$$

 Observe that for $n\geq 1$
\begin{eqnarray*}
S_{T_n}+\log(R+\sum_{i=1}^{T_n}|B_i|)&\leq & S_{T_n}+\log\Big(\sum_{j=1}^n (R+\sum_{i=T_{j-1}+1}^{T_{j}}|B_i|)\Big)\\
&\leq& \sum_{j=1}^n \bigg((S_{T_j}- S_{T_{j-1}})+\log\Big(1+R+\sum_{T_{j-1}+1\le i\le  T_{j}}|B_i|\Big)\bigg) \leq \sum_{j=1}^nV_j
\end{eqnarray*}
We claim that the variables $V_j$ are integrable. In fact since
$Y_i=\log A_i$ has a moment of order $2+\eps$, then classical
results guarantee that  $S_{T_j}- S_{T_{j-1}}$ is integrable and
$T_{j}-T_{j-1}$ has a moment of order $1/(2+\eps)$. So we need
only to prove that the variable $W_j$ has the first moment (see
\cite{CKW}, page 1279). By the Borel-Cantelli Lemma it sufficient
to show that
$$\limsup_{n\to\8}\frac 1n W_n <M \quad \mbox{a.s.}$$
for some constant $M$. We have
$$\frac 1n W_n=\frac{\sum_{j=1}^n(T_{j}-T_{j-1})^{1/(2+\eps)}}{n}\cdot \frac{W_n}{\sum_{j=1}^n(T_{j}-T_{j-1})^{1/(2+\eps)}}$$
By the strong law of large numbers the first term converges. For
the second term we have
$$\left(\frac{W_n}{\sum_{j=1}^n(T_{j}-T_{j-1})^{1/(2+\eps)}}\right)^{2+\eps}\leq \frac{W_n^{2+\eps}}{T_n}
\leq\frac{\sum_{k=1}^{T_n}\log(1+R+|B_k|)^{2+\eps}}{T_n}$$ which
converges since $(\log^+|B_1|)^{2+\eps}$ is integrable.

Let $U(y,x) = \sum_{n=1}^\8\E\Big[{\bf 1}_{(y,x]}\big( \sum_{i=1}^n V_i \big) \Big]$. Since $0<\E V_1<\8$, by renewal theorem
$$\lim_{x\to \8}\frac{U(0,x)}x = \frac 1{\E V_1}>0.$$
hence for any $m>1$ there exist large $N$ such that $\inf_{k\geq
N} \frac{U(m^k,m^{k+1})}{m^k}=C_1>0$. Therefore,

\begin{eqnarray*}
\int_{\R^d}f(|x|)\nu(dx) &\ge & C_R \E\bigg[
\sum_{n=0}^{\8} f\Big( \e^{\sum_{i=1}^nV_i}\Big)\bigg] \ge C_R \sum_{k>N}U(m^k,m^{k+1}) f\big(\e^{m^{k+1}}\big)\\
 &\ge& C_R  C_1 \sum_{k>N} m^k  f\big( \e^{m^{k+1}}\big) \ge
\frac{C_R C_1 }{m^2} \sum_{k>N} \int_{m^{k+1}}^{m^{k+2}} f(\e^x)dx \ge  C \int_{m^{N+1}}^\8 f(\e^x)dx,
\end{eqnarray*}
that proves \eqref{nu-leb}.

\noindent
{\bf Step 2.} Suppose that $\limsup_{z\to\8} \nu\big\{u:\;  z\alpha < |u|\le  z\beta  \big\}=0$, that is
for any fixed small $\eps>0$ there exists $N$ such that
$$\nu\Big\{u:\;  \frac{\b^k}{\a^k}\a < |u|\le  \frac{\b^k}{\a^k}\beta  \Big\}<\eps$$
for every $k>N$. Consider now the functions $f_n={\bf
1}_{\big[0,\frac{\b^{n+1}}{\a^n}\big]}$ on $\R_+$. Observe that
for $n>N$
\begin{eqnarray*}
\int_{\R^d}f_n(|u|)\nu(du) &=& \int_{\R^d}f_N(|u|)\nu(du) +\sum_{k=N+1}^{n}\int_{\R^d}{\bf 1}_{\big(\frac{\b^k}{\a^k}\a,\frac{\b^k}{\a^k}\beta\big]}(|u|)\nu(du)\\
&\leq&  \int_{\R^d}f_N(|u|)\nu(du)+ \eps(n-N).
\end{eqnarray*}
Thus $\limsup_{n\to\8}\frac{1}{n} \int_{\R^d}f_n(|u|)\nu(du)<\eps$, that is $\limsup_{n\to\8}\frac{1}{n} \int_{\R^d}f_n(|u|)\nu(du)=0$ since $\eps$ is arbitrary small.

On the other hand
$$\limsup _{n\to\8}\frac{1}{n}\int_M^\8 f_n(a)\frac{da}{a}= \limsup_{n\to\8}\frac{1}{n}\Big(\log(\frac{\b^{n+1}}{\a^n})-\log M\Big)=\log({\b}/{\a})>0$$
which leads to a contradiction with (\ref{nu-leb}).
\end{proof}

\appendix
\section{Proof of Theorem \ref{ps} and related corollaries}
\label{port-stone}

To prove Theorem \ref{ps} we will present a few consecutive
lemmas.

\begin{lem}
\label{lemma-ext} If $\psi\in\mathcal{F}(\ov\mu)$
the limit
$$ A \psi(x) = \lim_{\l\nearrow 1}A^\l\psi(x)$$
exists and is a continuous function. There exists a constant $C$ such that for all $x\in\R$ and $\l\in(1/2,1]$
$$\left|A^\l\psi(x)\right|\le C(1+x^2)$$
Furthermore for any fixed $x$
$$
\lim_{y\to\pm\8} \big(A\psi(x-y) - A\psi(-y)\big) = \mp x J(\psi)\s^{-2},
$$
and if $J(\psi)=0$ then
$$\lim_{x\to\pm\8} A\psi(x) = \mp  K(\psi)\s^{-2}.$$
\end{lem}
\begin{proof}
Observe that the Fourier transform of the measure $G^\l$ is
$$\wh{G^\l}(\th)=\sum_{n=0}^\infty\l^n\wh{\ov\mu}(\th)^n=\frac 1 {1-\l\wh{\ov\mu}(\th)}$$
Since $G^\l$ is a finite measure and $\wh{\psi}\in L^1(\R)$, using the Fubini Theorem, one has
\begin{eqnarray*}
G^\l*\psi(x)&=&\int_\R\psi(x+y)G^\l(dy)
=\frac 1{2\pi}\int_\R\int_\R \e^{-i\th(x+y)}\wh{\psi}(\th)d\th\, G^\l(dy)\\
&=&\frac 1{2\pi}\int_\R \e^{i\th x}\wh{\psi}(-\th)\int_\R \e^{i\th y}G^\l(dy)\,d\th
=\frac 1{2\pi}\int_\R \e^{i\th x}\wh{\psi}(-\th)\wh{G^\l}(\th)d\th\\
&=&\frac 1{2\pi}\int_\R \e^{i\th x}\frac{\wh{\psi}(-\th)}{1-\l\wh{\ov\mu}(\th)}d\th
\end{eqnarray*}

Notice that :
$$\frac{\wh{\psi}(-\th)}{1-\l\wh{\ov\mu}(\th)}=\frac{J(\psi)-iK(\psi)\th}{1-\l\wh{\ov\mu}(\th)}{\bf 1}_{[-a,a]}(\th)+\psi^\l_0(\th)$$
where $$\psi^\l_0(\th)= \frac{O(\th^2)}{1-\l\wh{\ov\mu}(\th)}{\bf
1}_{[-a,a]}(\th)+\frac{\wh{\psi}(-\th)}{1-\l\wh{\ov\mu}(\th)}{\bf
1}_{[-a,a]^c}(\th).$$ Moreover,
\begin{equation}
\label{j1}
|1-\l \wh{\ov\mu}(\th)|\ge \l|1-\wh{\ov\mu}(\th)|\ \mbox{ and }\   |1-\l \wh{\ov\mu}(\th)| \ge\l c|\th|^2
\end{equation}
for every $\th\in [-a,a]$. Therefore, for $1/2\le \l\leq 1$
\begin{equation}\label{eq-bound-psi}
\left|\psi^\l_0(\th)\right|<  C\left({\bf 1}_{[-a,a]}(\th)+\left|\frac{\wh{\psi}(-\th)}{1-\wh{\ov\mu}(\th)}\right|{\bf 1}_{[-a,a]^c}(\th)\right)\in L^1(d\th)
\end{equation}

Take $\phi\in\mathcal{F}(\ov\mu)$ such that $J(\phi)=J(\psi)$ then
\begin{equation*}
\begin{split}
%\label{d1}
G^\l*\phi(-y)-G^\l*\psi(x-y) &= \frac 1{2\pi}\int_\R \frac{ \e^{-iy\th} \wh \phi(-\th) -   \e^{i(x-y)\th}\wh\psi(-\th)}{1-\l\wh{\ov\mu}(\th)}d\th\\
&= \frac 1{2\pi}\int_{|\th|<a}\frac{
\e^{-iy\th}\big((J(\phi)-i\th K(\phi)) -\e^{ix\th}(J(\psi)-i\th K(\psi))\big)
}{1-\l\wh{\ov\mu}(\th)}d\th \\
&\quad+\frac 1{2\pi}\int_{|\th|\ge a}
\e^{-iy\th}(\phi^\l_0(\th)-\e^{ix\th}\psi^\l_0(\th))d\th .
\end{split}
\end{equation*}

The first integral can be decomposed as
\begin{align*}
&\frac 1{2\pi}\int_{|\th|<a}\frac{\e^{-iy\th}
\left((J(\psi)-i\th K(\phi)) -\e^{ix\th}(J(\psi)-i\th K(\psi))\right)
}{1-\l\wh{\ov\mu}(\th)}d\th  \\
&=\frac 1{2\pi}\int_{|\th|<a}\frac{\e^{-iy\th} \left((J(\psi)-i\th K(\phi))
 -(1+ix\th + (\e^{ix\th}-1-ix\th))(J(\psi)-i\th K(\psi))\right)} {1-\l\wh{\ov\mu}(\th)}d\th\\
&= -\frac{i}{2\pi} (K(\phi)- K(\psi)+xJ(\psi) ) \int_{|\th|<a}
\frac{\e^{-iy\th}\th}{1-\l\wh{\ov\mu}(\th)}d\th \\ &
\quad+  \int_{|\th|<a}
\frac{\e^{-iy\th}\left((J(\psi)-i\th K(\psi))(1+ix\th-\e^{ix\th})+x\th^2K(\psi)\right)}{1-\l\wh{\ov\mu}(\th)}d\th.
\end{align*}
Let
$$C^\l_{\ov\mu}(\th)(y) =\frac{i}{2\pi}\int_{|\th|<a}
\frac{\e^{-iy\th}\th}{1-\l\wh{\ov\mu}(\th)}d\th .$$ By Theorem
3.1'' in \cite{PS1} as $\l$ goes to 1 the limit
$$ \lim_{\l\nearrow1}C^\l_{\ov\mu}(y)=C^1_{\ov\mu}(y)$$
exists, it is finite and $\lim_{y\to\pm\infty}C^1_{\ov\mu}(y)=\pm
\s^{-2}$. Summing up we may write
\begin{equation}\label{eq-dec-Gl}
G^\l*\phi(-y)-G^\l*\psi(x-y)=-\left(K(\phi)-
K(\psi)+xJ(\psi)\right)   C^\l_{\ov\mu}(y)+\int_\R
\e^{-iy\th}h^\l_{\psi,\phi,x}(\th)d\th ,
\end{equation}
where the functions
$$h^\l_{\psi,\phi,x}(\th)= \frac{1}{2\pi}\left({\bf 1}_{[|\th|<a]}
\frac{(J(\psi)-i\th K(\psi))(1+ix\th-\e^{ix\th})+x\th^2K(\psi)}{1-
\l \wh{\ov\mu}(\th)}
+\phi^\l_0(\th)-\e^{ix\th}\psi^\l_0(\th)\right)$$ are bounded by a
function in $L^1(d\th)$ uniformly for all $\l\in [1/2,1]$ and $x$
in compact sets. More precisely by \eqref{j1} and
\eqref{eq-bound-psi}, there is an integrable function
$H=H_{\phi,\psi}$ such that for all $x\in\R$ and $\l\in[1/2,1]$:
\begin{equation}\label{eq-bound-h}
|h^\l_{\psi,\phi,x}(\th)|\le C(1+x^2)H(\th).
\end{equation}

By  the Lebesgue's dominate convergence theorem, the following limit exists
\begin{align}\label{eq-lim-G}
\lim_{\l\nearrow 1}G^\l*\phi(-y)-G^\l*\psi(x-y)=&-\left(K(\phi)- K(\psi)+xJ(\psi)\right)   C^1_{\ov\mu}(y) +\int_\R \e^{-iy\th}h^1_{\psi,\phi,x}(\th)d\th
\end{align}

For $\phi=J(\psi)g$ and $y=0$, we have $A^\l\psi(x)=G^\l*\phi(0)-G^\l*\psi(x)$ thus
\begin{align*}
A^\l\psi(x)=&-\left(J(\phi)K(g)- K(\psi)+xJ(\psi)\right)
C^\l_{\ov\mu}(0) +\int_\R h^\l_{\psi,\phi,x}(\th)d\th
\end{align*}
and
$$A\psi(x)=\lim_{\l\nearrow 1}A^\l \psi(x)=-\left(J(\phi)K(g)- K(\psi)+xJ(\psi)\right)   C^1_{\ov\mu}(0) +\int_\R h^1_{\psi,\phi,x}(\th)d\th .$$
Hence we have proved the existence of the recurrent potential
kernel. The continuity of $A\psi$ follows from uniform
integrability. Furthermore by \eqref{eq-bound-h} and since
$C^1_{\ov\mu}(0)$ is finite, we also have
$$\left|A^\l\psi(x)\right|\le C'(1+x^2)$$

Take now $\phi=\psi$ then by \eqref{eq-dec-Gl} and \eqref{eq-lim-G}
\begin{align*}
A\psi(x-y)-A\psi(-y)&= \lim_{\l\nearrow 1}G^\l*\psi(-y)-G^\l*\psi(x-y)
\\&=-xJ(\psi) C^1_{\ov\mu}(y) +   \wh{h^1_{\psi,\psi,x}}(-y).
\end{align*}
Since $ h^1_{\psi,\psi,x}\in L^1(d\th)$, then $\wh{h^1_{\psi,\psi,x}}$ is in $C_0(\R)$ for any fixed $x$.
Thus
$$\lim_{y\to\pm\infty}A\psi(-y)-A\psi(x-y)
=-xJ(\psi)\lim_{y\to\pm\infty}C_{\ov\mu}(y)
=\mp xJ(\psi)\s^{-2}.$$

If $J(\psi)=0$ then $A^\l\psi=-G^\l*\psi$ and we can take $\phi=0$
and $x=0$. Thus by (\ref{eq-lim-G})
$$
A\psi(-y)=K(\psi)C_{\ov\mu}(y) + \wh{h}^1_{\psi,0,0}(-y)
$$
and passing with $y$ to $\pm\8$ we obtain the expected limit.
\end{proof}
\begin{lem}
\label{lem-bound-A }
Assume $J(\psi)\geq 0$. For all $1>\eps>0$ there exists a constant $M$ such that for  any $x\in \R$
$$ -M+(1-\eps)J(\psi)\s^{-2}|x|\leq A\psi(x)\leq M+ (1+\eps)J(\psi)\s^{-2}|x|.$$
\end{lem}
\begin{proof}
  If $J(\psi)=0$ the previous lemma guarantees that $A\psi$ is continuous and has limit at infinity. Hence $A\psi$ is bounded.

  Suppose $J(\psi)>0$ and fix $\eps>0$. By the previous Lemma there is $K>0$ such that for all $y>K$
  $$ (1-\eps)J(\psi)\s^{-2}<A\psi(y+1)-A\psi(y) <(1+\eps)J(\psi)\s^{-2}$$
  and for $y<-K$
  $$ (1-\eps)J(\psi)\s^{-2}<A\psi(y-1)-A\psi(y) <(1+\eps)J(\psi)\s^{-2}.$$
Since $A\psi$ is continuous,
$$\sup_{|x|\leq K+1}|A\psi(x)|=M'<\infty .$$
Set $M=M'+(1+\eps)J(\psi)\s^{-2}(K+1)$
then the bound holds for $|x|\leq K$.

For $x>K$ let $[x-K]$ be the integer part of $x-K$ then, since
$$A\psi(x)-A\psi(x-[x-K])=\sum_{i=0}^{[x-K]-1}\Big(A\psi(x-[x-K]+ i+1))-A\psi(x-[x-K]+i)\Big)$$
and  $x-[x-K]+i>K$ then
$$ (1-\eps)J(\psi)\s^{-2}[x-K]<A\psi(x)-A\psi(x-[x-K])<(1+\eps)J(\psi)\s^{-2}[x-K].$$
Thus
\begin{align*}
A\psi(x)&=A\psi(x-[x-K])+A\psi(x)-A\psi(x-[x-K])<M'+(1+\eps)J(\psi)\s^{-2}[x-K]
\\&<M+(1+\eps)J(\psi)\s^{-2}x
\end{align*}
and
\begin{align*} A\psi(x)&=A\psi(x-[x-K])+A\psi(x)-A\psi(x-[x-K])>-M'+(1-\eps)J(\psi)\s^{-2}[x-K]
\\&>-M+(1-\eps)J(\psi)\s^{-2}x.
\end{align*}
In the same way we prove the bound for $x<-K$.
\end{proof}

\begin{proof}[Proof of Theorem \ref{ps}]
In view of Lemma \ref{lemma-ext} the potential $A\psi$ is well
defined. To prove that $A\psi$ is a solution of the Poisson
equation observe that
$$\ov\mu*A^\l\psi=c_\l  J(\psi)-\sum_{n=0}^\infty\l^n\ov\mu^{*n+1}*\psi=A^\l\psi+G^\l*(\psi-\ov\mu*\psi).$$
Notice that
$$G^\l*(\psi-\ov\mu*\psi)(x)=\frac 1{2\pi}\int_\R \e^{ix\th}\wh\psi(-\th)\frac{1-\wh{\ov\mu}(\th)}{1-\l\wh{\ov\mu}(\th)}d\th,$$
and, by (\ref{j1}), the integrand is dominated by $2|\wh\psi|\in
L^1(d\th)$ for all $1/2<\l\leq 1$. Therefore,  by Lebesgue's
dominated converge theorem
$$\lim_{\l\nearrow1}\ov\mu*A^\l\psi=A\psi +\psi.$$

By Lemma \ref{lemma-ext} there exists $C$ such  for all $\l$,  $x\in \R$ and any fixed $x$ :
$$|A^\l\psi(x+y)|\le C(1+(x+y)^2)\in L^1(\ov\mu(dy)),$$
then by dominate convergence
$$\ov\mu*A\psi(x)=\int_\R\lim_{\l\nearrow 1}A^\l\psi(x+y)\ov\mu(dy)=\lim_{\l\nearrow 1}\int_\R A^\l\psi(x+y)\ov\mu(dy)
=A\psi(x)+\psi(x).$$

The limit behavior is a direct consequence of Lemmas \ref{lemma-ext} and  \ref{lem-bound-A }.
\end{proof}

\begin{proof}[Proof of Corollary \ref{ps1}] First suppose $J(\psi)=0$.
Assume that $f$ is a continuous   solution of the Poisson
equation. Since
\begin{eqnarray*}
\ov\mu* f &=& f+\psi,\\
\ov\mu* A\psi &=& A\psi+\psi,
\end{eqnarray*}
the function $h=f-A\psi$ is $\ov\mu$-harmonic. It is bounded from
below because both   $-A\psi$ and $f$ are bounded from below.
Therefore by the Choquet-Deny theorem \cite{D} one has $h(x+y)=
h(x)$ for all $y$ in the closed subgroup generated by the support
of $\ov\mu$.
\medskip

Conversely, suppose that there exists a bounded solution $f_0$ of
the Poisson equation. Then $A\psi-f_0$ is $\ov\mu$-harmonic and
bounded from below, and so the Choquet-Deny theorem implies that
$A\psi$ is bounded. Thus
$$\lim_{x\to\infty}\frac{A\psi(x)}{x}=0$$
 and  by \eqref{eq-lim-A/x},
  we deduce $J(\psi)=0$.
\end{proof}

\begin{proof}[Proof of Corollary \ref{cor-form-solution}]
We apply a results of \cite{Ba}. First observe that the measure
$\ov\mu$ is strictly aperiodic (in the sense of \cite{Ba}) on
$G(\ov\mu)$. Consider the function
$$h(x) =\left(\frac{\sin x}{x}\right)^4.$$ Clearly,
$$\wh h(x)=C_1 {\bf 1}_{[-1,1]}*{\bf 1}_{[-1,1]}*{\bf 1}_{[-1,1]}*{\bf 1}_{[-1,1]}(x);$$
so it compactly supported. Thus the function $h$ belongs to the
class of functions $\mathcal{F}$ defined in \cite{Ba} (see also
\cite{PS2}).

Let $m$ be the Haar measure of the $G(\ov\mu)$, that is $m$ is either the Lebesgue measure on $\R$ or the counting measure. Then
$$\int_{G(\ov\mu)} |A\psi(x) h(x)| m(dx)\leq C \int_{G(\ov\mu)} (1+|x|) h(x) m(dx)<+\infty.$$
In view of Corollary 2 of \cite{Ba}, we conclude.
\end{proof}

\begin{proof}[Proof of Lemma \ref{phi0}]
To prove the first part of the Lemma it is enough to establish the following formula
\begin{equation}
\label{lh1}
 r (x) = \left\{
\begin{array}{ccc}
-2\E[(Y+x){\bf 1}_{Y+x\le 0}]&\ \mbox{ for }& x\ge 0\\
2\E[(Y+x){\bf 1}_{Y+x>0}]&\ \mbox{ for }& x < 0\\
\end{array}
\right.
\end{equation}
Indeed, since $\E Y=0$, for $x\ge 0$ we write
\begin{multline*}
 r (x) = \E[(Y+x){\bf 1}_{Y+x\geq 0}-(Y+x){\bf 1}_{Y+x< 0}-x]\\
 =\E[(Y+x) - 2(Y+x){\bf 1}_{Y+x< 0}-x]
  =-2  \E[(Y+x){\bf 1}_{Y+x < 0}].
\end{multline*}
For $x<0$ we proceed exactly in the same way. By \eqref{lh1}, the
function $ r $ is nonnegative.
\medskip

The Fourier transform can be computed in distribution sense. Let $a(x) = |x|$,
then  $ r  = (\ov\mu-\d_0)*a$ and $\wh a(\th) = \frac C{\th^2}$, hence $\wh r(\th)=C\cdot \frac{\wh{\ov\mu}-1}{\th^2}$.
\medskip

To prove \eqref{as1}, by \eqref{lh1}, for $x\ge 0$ we write
%\begin{eqnarray*}
$$| r (x)| = 2\E[|Y+x| {\bf 1}_{Y+x<0}] = 2\int_{x+y<0}|x+y|\ov\mu(dy) \\
\le 2\int_\R |x+y|\e^{-\d_0(x+y)}\ov\mu(dy)\le C \e^{-\d_1 x},
$$%\end{eqnarray*}
for some constants $\d_1<\d_0<\d$.

\medskip

Assume now $\E|Y|^{4+\eps}<\8$. Then for $t>0$ we have
\begin{eqnarray*}
|r (x)| &=& 2\int_{y<-x} |y+x|\ov\mu(dy)
= 2\cdot \sum_{m=1}^\8 \int_{-(m+1)x\le y < -mx}|y+x|\ov\mu(dy)\\
&\le& 2\cdot \sum_{m=1}^\8 mx \int_{|y|>mx} \ov\mu(dy) \le
2\cdot \sum_{m=1}^\8 mx \int_\R \frac{|y|^{\chi}}{m^{4+\eps}x^{4+\eps}}\ov\mu(dy) \le \frac C{x^{3+\eps}}.
\end{eqnarray*}

It is clear that if $\E|Y|^{4+\eps}<\8$ then $r\in \mathcal{F}(\ov\mu)$.
If $\psi=r*\zeta$ with $\zeta$ and $x^2\zeta$ in $L^1(\R)$ then  it is easily checked that both $\psi$ and $x^2\psi$ are integrable. Since
$\wh \psi=\wh r \wh \zeta=C \frac{\wh{\ov\mu}-1}{\th^2}\wh \zeta$
and $\wh \zeta$ vanish at infinity then $\psi\in \mathcal{F}(\ov\mu)$
\end{proof}

\bibliographystyle{alpha}
\newcommand{\etalchar}[1]{$^{#1}$}

\end{document}